\documentclass[11pt, reqno]{amsart}
\usepackage{amscd,amsfonts,amsmath,amssymb,amstext,amsthm}
\usepackage{url}
\usepackage[all]{xy}
\usepackage{tikz-cd}
\usepackage{xcolor}
\usepackage{enumerate}
\usepackage{verbatim}
\usepackage[linktocpage=true]{hyperref}
\usepackage{enumitem}
\usepackage{mathtools}

\makeatletter

\theoremstyle{plain}
\newtheorem{theorem} {Theorem} [section]
\newtheorem{lemma} [theorem]{Lemma}
\newtheorem{proposition}[theorem]{Proposition}
\newtheorem{corollary} [theorem]{Corollary}

\theoremstyle{definition}

\newtheorem{remark}[theorem]{Remark}
\numberwithin{equation}{section}

\title[$L^2$-Dolbeault Cohomology isomorphism]{Cohomology isomorphism of symmetric power of cotangent bundle of ball quotient and its toroidal compactification}
\date\today
\author{Seungjae Lee and Aeryeong Seo}
\address{Center for Complex Geometry, Institute for Basic Science, 55,
Expo-ro, Yuseong-gu, Daejeon, 34126, South Korea}
\email{seungjae@ibs.re.kr}

\address{Department of Mathematics,
Kyungpook National University,
Daegu 41566, South Korea}%
\email{aeryeong.seo@knu.ac.kr}


\begin{document}
\maketitle
\begin{abstract}
In this paper, we investigate the $L^2$-Dolbeault cohomology of the symmetric power of cotangent bundles of ball quotients with finite volume, as well as their toroidal compactification. Through the application of Hodge theory for complete hermitian manifolds, we establish the existence of Hodge decomposition and Green's operator. 
Moreover, we extend the results
by Adachi \cite{A21} and Lee--Seo \cite{LS23-2} from compact complex hyperbolic spaces to complex hyperbolic spaces with finite volume.
\end{abstract}

\section{Introduction}
For a complex manifold $X$, let $S^mT_X^*$ denote the $m$-th symmetric power of the holomorphic cotangent bundle of $X$. Our main theorem in this paper is
\begin{theorem}\label{main}
Let $\mathbb B^n$ be the unit ball in $\mathbb C^n$.
Let $\Gamma$ be a torsion-free lattice of $\text{Aut}(\mathbb{B}^n)$ with only unipotent parabolic automorphisms. Let $\Sigma = \mathbb{B}^n/ \Gamma$ be a quotient of $\mathbb B^n$ with finite volume and $\overline \Sigma$ be its toroidal compactification.
Then for each $m,r,s \in \mathbb{N} \cup \{0 \}$ there exists a holomorphic vector bundle $E_{r,m}$ over $\overline \Sigma$ such that if $r=0,n$ or $m \geq n-1$, then
\begin{equation*}
H^{r,s}_{L^2, \bar \partial} ( \Sigma, S^m T_{\Sigma}^*)\cong   
H^{s} \big( \overline{\Sigma}, E_{r,m}  \big)
\end{equation*}
where $H^{r,s}_{L^2, \bar \partial} ( \Sigma, S^m T_{\Sigma}^*)$ is the $L^2$-Dolbeault cohomology group of $S^mT_\Sigma^*$ over $\Sigma$ with respect to the metric induced from the Bergman metric on $\mathbb B^n$.

\end{theorem}

The requirement in Theorem \ref{main} that $\Gamma$ has only unipotent parabolic automorphisms is given for the smoothness of its toroidal compactification. In fact, every torsion-free lattice of $\textup{Aut}(\mathbb{B}^n)$ has a finite index subgroup having only unipotent parabolic automorphisms. Hence, any complex hyperbolic space form with finite volume has a finite covering which has a smooth toroidal compactification.

In \cite[Theorem 3.1]{HLWY22}, the authors established a result for a compact K\"ahler manifold $X$, a line bundle $L$ over $X$, and its simple normal crossing divisor $D$. They showed that there exists a hermitian metric $h^L$ on $X-D$ such that the sheaf cohomology $H^s(X, \Omega^r(\log D)\otimes \mathcal O(L))$ is isomorphic to the $L^2$-Dolbeault cohomology $H^{r,s}_{L^2, \bar\partial}(X-D, L, \omega_P, h^L_{X-D})$, provided that $\omega_P$ is a K\"ahler metric of the Poincar\'e type on $X-D$. 
It is worth mentioning that the induced metric $\omega$ on $\Sigma$ from the Bergman metric on $\mathbb B^n$ is not a Poincar\'e type in its toroidal compactification. 
This difference requires a more complex analysis in our paper to prove Theorem~\ref{main}.

The key ingredient in the proof of Theorem \ref{main} is to establish an $L^2$-version of Dolbeault-Grothendieck lemma for $S^m T_{\Sigma}^*$ using the induced metric $\omega$. 
One of the main difficulties in establishing this lemma is that the bundle $S^m T_{\Sigma}^* \otimes K_{\Sigma}^{-1}$ is not Nakano semi-positive for small $m$. If it were, one could apply the Bochner-Kodaira-Nakano inequality (for example, \cite[Chapter VIII (4.2)]{Dem}) to $\Lambda^{(r,0)} T^*_{\Sigma} \otimes S^m T_{\Sigma}^* \otimes K_{\Sigma}^{-1}$-valued $(n,s)$-forms.
To overcome this difficulty, { when $r=0$ or $n$}, we utilize a locally quasi-isometric metric relative to $\omega$ to establish vanishing of certain $L^2$-Dolbeault cohomologies for $S^m T_{\Sigma}^*$, which is then used to obtain the desired conclusion.
For { $1 \leq r \leq n-1$}, we apply the strategy presented in \cite{Ch11}: By using Berndtsson-Charpentier's proof of a version of Donnelly-Fefferman type estimate, along with suitably chosen weights, we are able to obtain an $L^2$-estimate, which is not easily obtained using the standard H\"ormander $L^2$-estimate.
As a result, we establish the local solvability of the $\bar \partial$-equation on $S^m T_{\Sigma}^*$-valued $(0,1)$-forms.

For a compact hermitian manifold $X$ and a holomorphic hermitian vector bundle $E$ over $X$, it is well known that the Dolbeault cohomology group $H^{r,s}(X,E)$ is finite dimensional. 
In the context of Hodge theory, it has been shown that the set of $E$-valued harmonic $(r,s)$-forms on $X$ is finite dimensional and this set is isomorphic to $H^{r,s}(X,E)$. 
However, when $X$ is non-compact, interesting phenomenon occur. For instance, in 1983, Donnelly-Fefferman showed that the dimension of the space of square integrable harmonic $(r,s)$-forms vanishes if $r+s\neq n$ and is infinite if $r+s=n$, when $X$ is a strictly pseudococonvex domain in $\mathbb C^n$ equipped with its Bergman metric. 
In line with this research, refer to \cite{DF83, O89, G91}.
In this paper, we are interested in the finite dimensionality of $L^2$-Dolbeault cohomology when $X$ is a complex hyperbolic space form $\Sigma$, i.e. a ball quotient with finite volume and $E$ is the symmetric power $S^mT^*_\Sigma$. By Theorem~\ref{main} we immediately obtain

\begin{corollary}
If $r=0,n$ or $m \geq n-1$, then   $\dim H^{r,s}_{L^2, \bar \partial} ( \Sigma, S^m T_{\Sigma}^*)<\infty$.
\end{corollary}

\medskip

Using Hodge theory methods for compact complex manifolds, Kashihara-Kawai~\cite{KK87} demonstrated the $L^2$-Hodge decomposition theorem. For a concise summary, we use Zhao's presentation in \cite{Z15}.

\begin{theorem}[\cite{Z15, KK87}]\label{KK}
    Let $(X,g)$ be a complete hermitian manifold of dimension $n$ and let $(E,h)$ be a hermitian vector bundle over $X$. For $0 \leq s \leq n$, suppose that $H^{r,s}_{L^2, \bar\partial}(X,E)$ are finite dimensional for all $s$. Then
    \begin{enumerate}
        \item The operator $\bar \partial\colon L^{r,s}_2(X,E)\to L^{r,s+1}_2(X,E)$ and its Hilbert adjoint $\bar\partial^*$ have closed images, and there is an orthogonal decomposition for each $q$,
        \begin{equation}\nonumber
            L^{r,s}_2(X,E) = \textup{Im}\, \bar\partial \oplus \textup{Im}\, \bar\partial^* \oplus \mathcal H^{r,s}_2(X,E),
        \end{equation}
        where $\mathcal  H^{r,s}_{L^2, \bar\partial}(X,E):= \ker \bar\partial\cap \ker \bar\partial^*$.
        As a consequence, there is an isomorphism
        $$H^{r,s}_{L^2, \bar\partial}(X,E)\cong \mathcal H^{r,s}_2(X,E).$$
        \item The operator $\Delta_{\bar \partial} = \bar \partial^* \circ \bar \partial + \bar \partial \circ \bar \partial^*$ is a self-adjoint operator acting on $L^{r,s}_{2} (X,E)$, and it satisfies
      $
       \bar \partial^* \circ \Delta_{\bar \partial}  = \Delta_{\bar \partial} \circ \bar \partial^*. 
        $
       \item Denoting $H$ be the projection operator $H\colon L^{r,s}_2(X,E)\to \mathcal H^{r,s}_2(X,E)$, then the Green operator 
       $G:= \left(\Delta_{\bar \partial}|_{\mathcal H^{r,s}_2(X,E)^\perp}\right)^{-1}(I -H)$ is well defined and bounded.
       Moreover, we have the following identity:
        $$
        \Delta_{\bar \partial} \circ G = G \circ \Delta_{\bar \partial} = I -H,\quad H \circ G=G \circ H=0.
       $$
  \end{enumerate}
\end{theorem}

By Theorem~\ref{main} and Theorem~\ref{KK}, we have 
\begin{theorem}\label{Hodge decomposition} 
For the unit ball $\mathbb B^n$ in $\mathbb C^n$, let $\Gamma$ be a torsion-free lattice of $\text{Aut}(\mathbb{B}^n)$ with only unipotent parabolic automorphisms and let $\Sigma = \mathbb{B}^n/ \Gamma$ be a quotient of $\mathbb B^n$ with finite volume. Let $g$ be the induced metric on $\Sigma$ from the Bergman metric on $\mathbb B^n$.  { If $r=0,n$ or $m \geq n-1$, then for each $s\in \mathbb N\cup\{0\}$}
    \begin{enumerate}
        \item The operator $\bar \partial\colon L^{r,s}_2(\Sigma,S^m T_{\Sigma}^*)\to L^{r,s+1}_2(\Sigma,S^m T_{\Sigma}^*)$ and its Hilbert adjoint $\bar\partial^*$ have closed images, and there is an orthogonal decomposition for each $s$,
        \begin{equation}\nonumber
            L^{r,s}_2(\Sigma,S^m T_{\Sigma}^*) = \textup{Im}\, \bar\partial \oplus \textup{Im}\, \bar\partial^* \oplus \mathcal H^{r,s}_2( \Sigma,S^m T_{\Sigma}^*),
        \end{equation}
        where $\mathcal  H^{r,s}_2(\Sigma,S^m T_{\Sigma}^*):= \ker \bar\partial\cap \ker \bar\partial^*$.
        As a consequence, there is an isomorphism
        $$
        H^{r,s}_{L^2, \bar\partial}(\Sigma,S^m T_{\Sigma}^*)\cong \mathcal H^{r,s}_2(\Sigma,S^m T_{\Sigma}^*).
        $$
        \item The operator $\Delta_{\bar \partial} = \bar \partial^* \circ \bar \partial + \bar \partial \circ \bar \partial^*$ is a self-adjoint operator acting on $L^{r,s}_2 (\Sigma,S^m T_{\Sigma}^*)$, and it satisfies
        $
        \bar \partial^* \circ \Delta_{\bar \partial} = \Delta_{\bar \partial} \circ \bar \partial^*. 
        $
        \item Denoting $H$ be the projection operator $H\colon L^{r,s}_2(\Sigma,S^m T_{\Sigma}^*)\to \mathcal H^{r,s}_2(\Sigma,S^m T_{\Sigma}^*)$, then the Green operator 
        $G:= \left(\Delta_{\bar \partial} |_{\mathcal H^{r,s}_2(\Sigma,S^m T_{\Sigma}^*)^\perp}\right)^{-1}(I -H)$ is well defined and bounded.
        Moreover, we have the following identity:
        $$
        \Delta_{\bar \partial} \circ G = G \circ \Delta_{\bar \partial } = Id -H,\quad H \circ G=G \circ H=0.
        $$
        
    \end{enumerate}
\end{theorem}

By applying Theorem~\ref{Hodge decomposition}, we extend the results presented in \cite{A21,LS23-2} from compact complex hyperbolic spaces to complex hyperbolic spaces with finite volume:  let $\Omega=\mathbb B^n\times\mathbb B^n / \Gamma$ be the quotient of $\mathbb{B}^n \times \mathbb{B}^n$ under the diagonal action $(z,w)\mapsto(\gamma z, \gamma w) $ for $\gamma\in\Gamma$.  Then $\Omega$ is a holomorphic $\mathbb B^n$-fiber bundle over $\Sigma$.  A K\"ahler form $\omega_{\Omega}$ on $\Omega$ is defined by
$$
\omega_{\Omega} := \frac{\sqrt{-1}}{n+1} \partial \bar \partial \log K_{\mathbb{B}^n} (z,z) + \frac{\sqrt{-1}}{n+1} \partial \bar \partial \log K_{\mathbb{B}^n} (w,w)
$$
and the volume form $dV_{\omega_{\Omega}}$ is defined by $ \frac{1}{(2n)!} \omega_{\Omega}^{2n}$. Given the volume form,  the volume of $\Omega$ is finite (see the proof of Lemma 4.14 of \cite{LS23-2}). 

Consider an automorphism of $\mathbb{B}^n$ 
$$
T_z (w) := \frac{z - P_z (w) - s_z Q_z(w)}{1- w \cdot \bar z}
$$
where $|z|^2 = z \cdot \bar z$ and $s_z = \sqrt{1-|z|^2}$, $P_z$ is the orthogonal projection from $\mathbb{C}^n$ onto the one-dimensional subspace $[z]$ generated by $z$, and $Q_z$ is the orthogonal projection from $\mathbb{C}^n$ onto $[z]^{\perp}$. For measurable sections $f,g$ on $\Lambda^{r,s} T_{\Omega}^*$, we define
$$
\langle \langle f, g \rangle \rangle_{\alpha} := \frac{\Gamma(n+\alpha+1)}{n! \Gamma (\alpha+1)} \int_{\Omega} \langle f, g \rangle_{\omega_{\Omega}} \delta^{\alpha+(n+1)} dV_{\omega_{\Omega}}
$$ 
where $\delta := 1- |T_z w|^2$. We define the weighted $L^2$-space by
$$
L^2_{(r,s),\alpha}(\Omega) := \{ f: f \; \text{is a measurable section on} \; \Lambda^{r,s} T_{\Omega}^* \; \text{satisfying} \; \| f \|^2_{\alpha} := \langle \langle f, f \rangle \rangle_{\alpha} <\infty \}
$$
for $\alpha >-1$, and define $A^2_{\alpha} (\Omega) := L^2_{(0,0),\alpha} (\Omega) \cap \mathcal{O}(\Omega)$. 
The \textit{Hardy space} is defined by
$$
A^2_{-1}(\Omega) = \{ f \in \mathcal{O}(\Omega) : \| f \|^2_{-1} := \lim_{\alpha \searrow -1} \| f \|^2_{\alpha} < \infty \}.
$$

\begin{theorem}\label{extension}
Let $\Gamma$ be a torsion-free lattice of the automorphism group of the complex unit ball $\mathbb B^n$ with only unipotent parabolic automorphisms and $\Sigma = \mathbb B^n / \Gamma$ be a complex hyperbolic space form with finite volume.
Let $\Omega = \mathbb B^n \times \mathbb B^n /\Gamma$ be a holomorphic $\mathbb B^n$-fiber bundle
under the diagonal action of $\Gamma$ on $\mathbb B^n\times\mathbb B^n$.  Then we have an injective linear map
$$
\Phi\colon \bigoplus_{m=0}^\infty  H^{0,0}_{L^2, \bar \partial} (\Sigma, S^m T^*_\Sigma)  
\rightarrow \bigcap_{\alpha>-1}
A^2_{\alpha}(\Omega) \subset \mathcal O(\Omega)
$$
having a dense image in $\mathcal O(\Omega)$ equipped with
compact open topology.
\end{theorem}

By apply a similar argument given in \cite[Corollary 4.19, Theorem 4.20]{LS23-2}, we obtain
\begin{corollary}\label{no nonconstant holo}
Let $\Sigma = \mathbb{B}^n /\Gamma$ be a hyperbolic space form with finite volume for a torsion-free lattice $\Gamma \subset \text{Aut}(\mathbb{B}^n)$. Let $\Omega$ be the quotient of $\mathbb{B}^n \times \mathbb{B}^n$ by the diagonal action of $\Gamma$. Then $A^2_{-1}(\Omega) \cong \mathbb{C}$ and $\Omega$ has no any non-constant bounded holomorphic function.
\end{corollary}

This paper is organized as follows:
in Section~\ref{preliminaries}, we recall the description of the toroidal compactification of the ball quotient with finite volume and $L^2$-Dolbeault cohomology for holomorphic vector bundles over complete K\"ahler manifolds.
In Section~\ref{metric description}, we introduce a quasi-isometric metric on the ball quotient, comparing it to the induced metric from the Bergman metric on the unit ball, and investigate some of its properties.
In Section~\ref{proof of main theorem}, we prove Theorem~\ref{main} and Theorem~\ref{Hodge decomposition}.
In Section~\ref{application}, we demonstrate Theorem~\ref{extension} by using the existence of Green's operator on $S^mT_\Sigma^*$. Additionally, we provide a generalization of the result presented in \cite{LS23-1} to the ball quotient of finite volume when the quotient admits a smooth toroidal compactification.

\medskip

{\bf Acknowledgement}
The first author was supported by the Institute for Basic Science (IBS-R032-D1). The second author was partially supported by Basic Science Research Program through the National Research Foundation of Korea (NRF) funded by the Ministry of Education (NRF-2022R1F1A1063038).

\section{Preliminaries}\label{preliminaries}
\subsection{Toroidal compactification of finite volume ball quotient}
In this section, we will recall the description of the toroidal compactifications for $\mathbb B^n/\Gamma$ given in \cite{Mok, Wong}.

Let $\Gamma$ be a torsion-free subgroup of $\textup{Aut}(\mathbb B^n)$ and let $\Sigma=\mathbb B^n/\Gamma$. If $\Sigma$ is of finite volume for the induced Bergman metric on $\mathbb{B}^n$, then there exists only finite number of cusps $b_1, \cdots, b_k$ in $\partial \mathbb{B}^n$ \cite{S60, BB66, SY82}. For each $b\in \partial\mathbb B^n$, 
let $N_b:=\{\varphi\in\text{Aut}(\mathbb B^n): \varphi(b)=b\}$ be the normalizer of $b$. 

Now fix $b\in B$ and let 
$$
c\colon \mathbb B^n\to S_n:=\left\{(z',z_n)\in \mathbb C^{n-1}\times \mathbb C: \textup{Im}\, z_n> \|z'\|^2\right\}
$$
be a Cayley transformation such that $c$ extends real analytically to $\mathbb B^n-\{b\}$ and $c|_{\partial\mathbb B^n-\{b\}} \to \partial S_n$ is a real analytic diffeomorphism.
For any $N\geq 0$, define an open subset of $S_n$ by 
$$
S^{(N)} := \left\{(z',z_n)\in \mathbb C^{n-1}\times \mathbb C : \text{Im}z_n > |z'|^2 +N \right\}.
$$
Consider a holomorphic map $\Psi\colon \mathbb C^{n-1}\times\mathbb C\to\mathbb C^{n-1}\times\mathbb C^* $ given by 
\begin{equation}\label{covering}
\Psi(z',z_n)=(z', e^{\frac{2\pi i z_n}{\tau}}):=(w', w_n)
\end{equation}
for some $\tau\in \mathbb R$ and let $G := \Psi(S_n)$, $G^{(N)}:=\Psi(S^{(N)})$.
Then $G$ and $G^{(N)}$ are total spaces of a family of punctured discs over $\mathbb C^{n-1}$.
Define $\widehat G$ and $\widehat G^{(N)}$ by adding $\mathbb C^{n-1}\times \{0\}$ to $G$ and $G^{(N)}$, respectively.
Then we have 
\begin{equation}\nonumber
\begin{aligned}
\widehat G &= \left\{(w', w_n)\in \mathbb C^{n-1}\times\mathbb C: |w_n|^2 < e^{-\frac{4\pi}{\tau}\|w'\|^2}\right\},\\
\widehat G^{(N)} &= \left\{(w', w_n)\in \mathbb C^{n-1}\times\mathbb C: |w_n|^2 < e^{-\frac{4\pi N}{\tau}} e^{-\frac{4\pi}{\tau}\|w'\|^2}\right\}.
\end{aligned}
\end{equation}

Let $W_b$ be the unipotent radical of $N_b$ whose elements act on $S_n$ as affine automorphisms. Let $U_b:=[W_b, W_b]$ whose elements act on $S_n$ as translations in $z_n$ direction.
Then $\Gamma\cap W_b$ acts on $S$ as a discrete group of automorphisms and $[\Gamma\cap W_b, \Gamma\cap U_b]=0$, which implies that the action of $\Gamma\cap W_b$ descends from $S_n$ to $S_n/(\Gamma\cap U_b)$.
Since $U_b$ is 1-dimensional, $U_b\cap\Gamma\cong\mathbb Z$ is generated by some $\tau\in U_b\cong\mathbb R$. This implies that if we choose such $\tau$ to define the holomorphic map $\Psi$, we have $S_n/(\Gamma\cap U_b)\cong \Psi(S_n)=G$ 
and $S^{(N)}/(\Gamma\cap U_b)\cong\Psi(S^{(N)})=G^{(N)}$. Moreover, there is a group homomorphism
$\pi\colon \Gamma\cap W_b\to\text{Aut}(G)$ such that 
$\Psi\circ \varphi = \pi(\varphi)\circ\Psi$ for any $\varphi\in \Gamma\cap W_b$.
If $\Gamma$ has only unipotent prabolic automorphisms, then the action of $\pi(\Gamma\cap W_b)$ on $\mathbb C^{n-1}\times\{0\}$ is a lattice of translation $\Lambda_b$. From now on, we assume that $\Gamma$ has only unipotent parabolic automorphisms.
Define $D_b:=\mathbb C^{n-1}\times\{0\}/\Lambda_b$ be a torus. The toroidal compactification $\overline \Sigma$ of $\Sigma$ is set-theoretically given by 
$$
\overline \Sigma = \Sigma \cup \bigcup_{b\in B}D_b.
$$

Define 
$$\Omega_b^{(N)}:= \widehat G^{(N)}/\pi(\Gamma\cap W_b).$$ 
Thus there exists an embedding $\Omega_b^{(N)}-D_b \hookrightarrow \Sigma$ for sufficiently large $N$, and
we have 
\begin{equation}\nonumber
    \Omega_b^{(N)} \supset 
    G^{(N)}/\pi(\Gamma\cap W_b)\cong S^{(N)}/(\Gamma\cap W_b).
\end{equation}
For sufficiently large $N$, $\Omega_b^{(N)} -D_b$ do not overlap in $\Sigma$ for $b\in B$.

Consider the trivial line bundle 
\begin{equation}\nonumber
\begin{aligned}
    \mathbb C^{n-1}\times\mathbb C&\to\mathbb C^{n-1}\\
    (w', w_n) &\mapsto w'
\end{aligned}
\end{equation}
and define a hermitian metric
$$
\mu(w,w):= \|w\|^2 := e^{\frac{4\pi}{\tau}|w'|^2} |w_n|^2,
$$
which has negative curvature.
Since $\widehat G^{(N)}$ can be expressed by
$\left\{w\in\mathbb C^n : \mu(w,w)<e^{-\frac{4\pi}{\tau}N}\right\}$, $\widehat G^{(N)}$ is a level set of the trivial line bundle under the metric $\mu$.
Then the quotient of the restriction of the trivial line bundle to $\widehat G^{(N)}$
gives a line bundle $L\to D_b$ with the induced hermitian metric $\overline \mu$ of negative curvature.
This implies that $D_b$ is an abelian variety and $\Omega_b^{(N)}$ is the tubular neighborhood of $L$ with $\overline\mu$-length $<e^{\frac{-2\pi}{\tau}N}$.

Let $\omega_{S_n}$ be a K\"ahler form of $S_n$ given by 
\begin{equation}\nonumber
    \omega_{S_n} = \sqrt{-1}\partial\bar\partial(-\log(\text{Im}\, z_n-|z'|^2)).
\end{equation}
Remark that $\|\partial\left(\log(\text{Im}\, z_n -|z'|^2)\right)\|_{\omega_{S_n}}\equiv 1$ (cf. \cite{LS23}).
For the local coordinate $(w', w_n)$ given in \eqref{covering} in $\Omega_b^{(N)}$, the divisor $D_b$ is defined by $\{w_n=0\}$. 
Let $\omega_{\Omega_{b}^{(N)}}$ denote the induced metric on $\Omega_{b}^{(N)}$ from $\omega_{S_n}$ near the cusp $b$.
Using the coordinate change in \eqref{covering}, 
by the relation 
\begin{equation}\nonumber
    \log\|w\| = \frac{2\pi}{\tau} \left(|z'|^2 - \text{Im}\, z_n\right),
\end{equation}
we have 
\begin{equation}\nonumber
    \begin{aligned}
    \omega_{\Omega_{b}^{(N)}}
    &=\sqrt{-1}\partial \bar \partial (-\log(-\log \| w \|) ) \\
        \end{aligned}
\end{equation}
and 
\begin{equation}\nonumber
    \|\partial(-\log(-\log \| w \|) ) \|_{\omega_{\Omega_b^{(N)}}} 
    \equiv 1.
\end{equation}
Denote $\rho(w)$ be the strictly plurisubharmonic function which is defined near the cusp $b$ induced by the K\"ahler potential $-\log(-\log \| w \|) $. We let $dV_{\omega_{S_n}}$ be the volume form of $S_n$ with respect to its Bergman metric. We define the euclidean volume form as $dV = \sqrt{-1}\partial\bar\partial|w|^2$.
Then as $\|w\|\to 0$, there exist constants $C_1$, $C_2>0$ such that 
\begin{equation}\nonumber
    \frac{C_1}{\|w\|^2 (-\log\|w\|)^{n+1}} dV
    \leq dV_{\omega_{S_n}} 
    \leq \frac{C_2}{\|w\|^2 (-\log\|w\|)^{n+1}} dV.
\end{equation}

\subsection{$L^2$-Dolbeault cohomology} 
In this section, we will review some results for $L^2$-Dolbeault cohomologies.  For the detail, see \cite{Dem}. 

Let $(X,\omega)$ be a K\"ahler manifold and $(E,h^{E})$ be a holomorphic vector bundle over $X$. Let $|\cdot|^2_{\omega}$ be the induced norm on $\Lambda^{r,s}T_{X}^*$ from $\omega$. For $E$-valued $(r,s)$ forms $u$ and $v$, we denote by $\langle u, v \rangle_{h^E, \omega}$ the inner product on $E \otimes \Lambda^{r,s} T_{X}^*$ induced from $h^E$ and $\omega$. Denote $\langle u, u \rangle_{h^E, \omega}$ by $|u|^2_{h^E, \omega}$. For any measurable $E$-valued $(r,s)$ forms $u$ and $v$, we define an $L^2$-inner product of $u$ and $v$ by
$$
\langle \langle u , v \rangle \rangle_{h^E, \, \omega} := \int_{X} \langle u, v \rangle_{h^E, \, \omega} dV_{\omega}
$$
Then the $L^2$-norm of $u$ is given by
$$
\|u \|^2_{h^E, \, \omega} := \int_{X} |u|^2_{h^E, \, \omega} dV_{\omega}.
$$
For simplicity, we write $ \| u \|^2_{h^E, \, \omega}$ as $\| u \|^2$ if there is no ambiguity. 
The $L^2$-space of $E$-valued $(r,s)$ forms on ${X}$ is defined by
$$
L^{r,s}_{2} (X,E, h^E, \omega) := \left\{ f : \text{$f$ is a measurable section of $E\otimes \Lambda^{r,s}T_{X}^*$ such that $\| u \|^2 <\infty $}  \right\}.
$$
If there is no danger of confusion we abbreviate it to $L_{2}^{r,s}(X, E)$.  Let $C^\infty_{c,(r,s)}(X, E)$ be the space of compactly supported $E$-valued $(r,s)$ forms. We extend the operator 
$
\bar \partial: C^{\infty}_{c,(r,s)} (X, E) \rightarrow C^{\infty}_{c,(r,s+1)}(X, E)$ 
to a closed densely defined linear operator 
$$
\bar \partial : L^{r,s}_{2} (X,E) \rightarrow L^{r,s+1}_{2} (X,E)
$$ 
by taking the maximal closed extension of $\bar \partial$. Then we have the Hilbert adjoint 
$$
\bar \partial^*: L_{2}^{r,s}(X,E) \rightarrow L_{2}^{r,s-1}(X,
E)
$$
of $\bar \partial$, which is also a closed densely defined linear operator. We denote  by
$$
H^{r,s}_{L^2, \bar \partial} (X,E, h^E, \omega) \quad \text{or} \quad H^{r,s}_{L^2, \bar \partial} (X, E)
$$
the $L^2$-Dolbeault cohomology group of $(X, E, h^E, \omega)$.

Let $\Theta (E)$ be the Chern curvature tensor of $(E,h^E)$ and $\Lambda_{\omega}$ be the adjoint of the left multiplication of $\omega$. We say that $(X,\omega)$ has a K\"ahler potential if there exists a smooth function $\varphi:X \rightarrow \mathbb{R}$ satisfying $\omega = \sqrt{-1} \partial \bar \partial \varphi$.

\begin{proposition}\label{vanishing by kahler potential}
Let $(X,\omega)$ be a $n$-dimensional complete K\"ahler manifold and $(E,h^E)$ be a holomorphic vector bundle with a smooth hermitian metric $h^E$. If $(X,\omega)$ has a K\"ahler potential $\varphi$ such that $\sup_X|\partial \varphi |^2_{\omega} < C$ for a constant $C>0$ and $[i \Theta(E), \Lambda_{\omega} ] \geq 0$ in bidegree $(r,s)$, then
$$
H^{r,s}_{L^2, \bar \partial} (X,E, h^E,\omega)=0 \quad \text{when} \quad r+s \not = n.
$$
\end{proposition}
\begin{proof}
Let $\tau$ be a smooth $(r,s)$ form on $X$ and $\psi$ be a smooth function on $X$ which will be chosen later. We denote by $ \bar \partial^*_{h^E}$ the Hilbert adjoint of $\bar \partial$ with respect to $(X,E, h^E, \omega)$ and denote by $\bar \partial^*_{h^E, \psi}$ the Hilbert adjoint of $\bar \partial$ with respect to $(X, E, h^Ee^{-\psi}, \omega)$. Let $\zeta=e^{\psi/2} \tau$. Then
\begin{equation}\label{upper basic estimate for (n,1)}
\begin{aligned}
&\int_{X} |\bar \partial \zeta |^2_{h^E, \omega} e^{-\psi} dV_{\omega}  + \int_{X} |\bar \partial^*_{h^E, \psi} \zeta |^2_{h^E, \omega} e^{-\psi} dV_{\omega} \\
&= \int_{X} |e^{-\psi/2} \bar \partial(e^{\psi/2} \tau)|^2_{h^E, \omega} dV_{\omega} + \int_{X} |e^{-\psi/2} \bar \partial^*_{h^E, \psi} (e^{\psi/2} \tau)|^2_{h^E, \omega} dV_{\omega}.
\end{aligned}
\end{equation}
Since
\begin{equation*}
\begin{aligned}
| e^{-\psi/2} \bar \partial (e^{\psi/2} \tau) |^2_{h^E, \omega}
&= \left| \bar \partial \tau + \frac{\bar \partial \psi}{2} \wedge \tau \right|^2_{h^E, \omega} \\
&= |\bar \partial \tau |^2_{h^E, \omega} + 2 \text{Re} \left\langle \bar \partial \tau, \frac{\bar \partial \psi}{2} \wedge \tau \right\rangle_{h^E, \omega} + \bigg | \frac{ \bar \partial \psi}{2} \bigg|^2_{\omega} \big| \tau |^2_{h^E, \omega} \\
& \leq \left(2+\frac{1}{d}\right) |\bar \partial \tau|^2_{h^E,\omega}+ \frac{1+d}{4}| \bar \partial \psi|^2_{\omega} | \tau |^2_{h^E,\omega}
\end{aligned}
\end{equation*}
and
\begin{equation*}
\begin{aligned}
| e^{-\psi/2} \bar \partial^*_{h^E, \psi} (e^{\psi/2} \tau) |^2_{h^E, \omega}
&= \bigg| \bar \partial^*_{h^E} \tau + \frac{\bar \partial \psi}{2} \lrcorner \tau \bigg|^2_{h^E,\omega} \\
&= |\bar \partial^*_{h^E} \tau |^2_{h^E, \omega} + 2 \text{Re} \left\langle \bar \partial^*_{h^E} \tau, \frac{\bar \partial \psi}{2} \lrcorner \tau \right\rangle_{h^E,\omega} + \bigg| \frac{\bar \partial \psi}{2} \lrcorner \tau \bigg|^2_{h^E,\omega} \\
& \leq \left(2+\frac{1}{d}\right) |\bar \partial^*_{h^E} \tau|^2_{h^E, \omega} + \frac{1+d}{4} | \bar \partial \psi |^2_{\omega} |\tau |^2_{h^E, \omega},
\end{aligned}
\end{equation*}
by \eqref{upper basic estimate for (n,1)} we obtain
\begin{equation}\label{upper estimate 1 for (n,0)}
\begin{aligned}
&\int_{X} |\bar \partial \zeta |^2_{h^E, \omega} e^{-\psi} dV_{\omega}  + \int_{X} |\bar \partial^*_{h^E, \psi } \zeta |^2_{h^E, \omega} e^{-\psi} dV_{\omega}\\
&\leq
\left(2+\frac{1}{d}\right) \left( \int_{X}  | \bar \partial \tau |^2_{h^E,\omega} dV_{\omega}+\int_{X}  | \bar \partial^*_{h^E} \tau|^2_{h^E,\omega} dV_{\omega}  \right)
+ \frac{1+d}{2} \int_{X} | \bar \partial \psi |^2_{\omega} |\tau|^2_{h^E, \omega} dV_{\omega}.
\end{aligned}
\end{equation}
On the other hand, by the basic estimate for $(X,E, h^E,\omega)$ (for example, see \cite{Dem})
\begin{equation}\label{lower}
\int_{X} |\bar \partial \zeta |^2_{h^E, \omega } e^{-\psi} dV_{\omega}
+ \int_{X} |\bar \partial^*_{h, \psi} \zeta |^2_{h^E, \omega} e^{-\psi} dV_{\omega}
\geq \int_{X} \langle [\sqrt{-1}\Theta(E) + \sqrt{-1}\partial \bar \partial \psi , \Lambda_{\omega}] \zeta, \zeta \rangle e^{-\psi} dV_{\omega}.
\end{equation}

Now, let $\varphi$ be a K\"ahler potential of $\omega$ and
$C:= \sup_{X} |\partial \varphi|^2_{\omega}<\infty$. If $r+s>n$, then we choose $\psi=  t \varphi$ with a constant $t>0$ which will be chosen later. Then,
\begin{equation*}
\begin{aligned}
&\int_{X} \langle [\sqrt{-1} \Theta(E) + \sqrt{-1} \partial \bar \partial \psi, \Lambda_{\omega}] \zeta, \zeta \rangle e^{-\psi}dV_{\omega} \geq t(r+s-n) \int_{X} |\tau|^2_{h^E, \omega} dV_{\omega} 
\end{aligned}
\end{equation*}
by $\omega= \sqrt{-1}\partial \bar \partial \varphi$. 
Hence by \eqref{upper estimate 1 for (n,0)} and \eqref{lower}, we have
\begin{equation*}
\begin{aligned}
&\left(2+\frac{1}{d}\right) \left( \| \bar \partial \tau \|^2_{h^E, \omega} + \| \bar \partial^*_{h^E} \tau \|^2_{h^E, \omega} \right )
+ \frac{C(1+d)}{2} t^2 \int_{X}  | \tau |^2_{h^E, \omega} dV_{\omega}  \geq t(r+s-n)  \int_{X} | \tau |^2_{h^E, \omega} dV_{\omega}
\end{aligned}
\end{equation*}
which induces
\begin{equation}\label{basic estimate by Kahler potential}
\begin{aligned}
 \| \bar \partial \tau \|^2_{h^E, \omega} + \| \bar \partial^*_{h^E} \tau \|^2_{h^E, \omega} 
& \geq \frac{ t(r+s-n) - \frac{C(1+d)}{2} t^2 }{2+\frac{1}{d}}  \int_{X} | \tau |^2_{h^E, \omega} dV_{\omega} \\
& \geq \frac{d(r+s-n)^2}{2C(1+d)(2d+1)} \int_{X} | \tau |^2_{h^E, \omega} dV_{\omega}
\end{aligned}
\end{equation}
by letting $t = \frac{r+s-n}{C(1+d)}$.

Since $\tau$ is an arbitrary chosen smooth $(r,s)$ form, by \eqref{basic estimate by Kahler potential}, the proposition is proved when $r+s>n$ by following the proof of Theorem 4.5 in Chapter VIII of \cite{Dem}. 

If $r+s<n$, then we choose $\psi=-t \varphi$ with a constant $t>0$ which will be chosen later. Then by a similar argument, we obtain the same conclusion. 
\end{proof}

Let $(w_1, \cdots, w_n)$ be a local coordinate system and $\{ e_k \}$ be a local orthonormal frame for $(E,h^E)$. Then the Chern curvature tensor $\Theta(E)$ is of the form
$$
\Theta(E) = \sum_{j,k,\lambda,\mu} c_{jk \lambda \mu} dw_j \wedge d \overline w_k \otimes e_\lambda^* \otimes e_\mu.
$$
We say $E$ is of Nakano positive (resp. Nakano semi-positive)  if 
$$
\sum_{jk\lambda \mu} c_{jk\lambda\mu} \tau_{j \lambda} \overline{\tau}_{k \mu} >0 \quad (\text{resp.} \geq 0)
$$
for any non-zero vector $\tau = \sum_{j,\lambda} \tau_{j \lambda} \frac{\partial}{\partial w_j} \otimes e_{\mu}$.

\begin{corollary}\label{basic estimate by Kahler potential 2}
Let $(X,\omega)$ be a $n$-dimensional complete K\"ahler manifold and $(E,h^E)$ be a holomorphic vector bundle with a smooth hermitian metric $h^E$. If $(X,\omega)$ has a K\"ahler potential $\varphi$ such that $\sup_X|\partial \varphi |^2_{\omega} < C$ for a constant $C>0$ and $E$ is Nakano semi-positive, then
$$
H^{n,s}_{L^2, \bar \partial} (X,E, h^E,\omega)=0 \;\; \text{for all} \;\; s \geq 1 .
$$
\end{corollary}
\begin{proof}
It is known that the Nakano semi-positivity of $E$ implies $[\sqrt{-1} \Theta(E), \Lambda_{\omega} ] \geq 0$ 
in bidegree $(n,s)$ (for example, see the proof of  \cite[Chapter VII, Lemma 7.2]{Dem}). Therefore, the corollary follows by Proposition \ref{vanishing by kahler potential}.
\end{proof}

\section{Ball quotient type metric}\label{metric description}
We say that two riemannian metrics $g$ and $g'$ are {\it quasi-isometric}, if there exist two positive constants $C_1$, $C_2$ such that
$$
C_1 g \leq g' \leq C_2 g
$$
and we denote by $g \sim g'$ if $g$ and $g'$ are quasi-isometric. 

The lemmas in this subsection are influenced by the idea used in \cite{Fu92}. Especially, the proof of Lemma~\ref{quasi-isometric} uses similar method exploited in Lemma 2.3 of \cite{Sa85}. 

\begin{lemma}\label{quasi-isometric} 
Let $\Gamma$ be a torsion-free lattice of $\text{Aut}(\mathbb{B}^n)$ with only unipotent parabolic automorphisms and let $\Sigma= \mathbb{B}^n/\Gamma$ be a ball quotient with finite volume. Let $b$ be a cusp of $\Gamma$. Then for any point $p \in D_b$, there exists a positive constant $\epsilon$ satisfying
\begin{enumerate}
\item $p=(0,\cdots,0) \in \mathbb{D}^{n} (\epsilon) \subset \subset \Omega_{b}^{(N)}$. 
\item the induced K\"ahler form $\omega_{\Omega_{b}^{(N)}}$ is quasi-isometric to
a K\"ahler form 
\begin{equation}\nonumber
\widetilde \omega_{\Omega_b^{(N)}} := \sum_{k=1}^{n-1} \frac{dw_k \wedge d \bar w_k}{(-\log \| w \|)} + \frac{dw_n \wedge d \bar w_n}{\|w \|^2 (-\log \| w \|)^2},
\end{equation}
on $\mathbb{D}^{n-1} (\epsilon) \times \mathbb{D}^* (\epsilon) $, where $(w_1, \cdots, w_n)$ are the euclidean coordinates.
\end{enumerate}

\end{lemma}

\begin{proof}
By 
$$
\partial \bar \partial \| w \| = \frac{2\pi}{\tau} \| w \| \partial \bar \partial |w'|^2 + \frac{\partial \|w \| \wedge \bar \partial \| w \|}{\|w \|},
$$
it follows that
\begin{equation*}
\begin{aligned}
\omega_{\Omega_b^{(N)}} &= \sqrt{-1} \partial \bar \partial (- \log (-\log \| w \|)) \\
&= \frac{1+\log \|w \|}{\|w\|^2 (-\log \|w \|)^2} \sqrt{-1} \partial \| w \| \wedge \bar \partial \| w \| + \frac{1}{\| w \| (-\log \| w \|)} \sqrt{-1} \partial \bar \partial \| w \| \\
&=\frac{1}{\| w \|^2 (- \log \| w \|)^2} \sqrt{-1} \partial \| w \| \wedge \bar \partial \| w \| + \frac{2 \pi} {\tau} \frac{1}{(-\log \|w \|)} \sqrt{-1} \partial \bar \partial | w' |^2.
\end{aligned}
\end{equation*}
Since 
$$
\partial \| w \|  = \frac{2\pi}{\tau} e^{\frac{2\pi}{\tau} |w'|^2} |w_n| \partial |w'|^2 + e^{\frac{2\pi}{\tau} |w'|^2} \partial |w_n| = \frac{2\pi}{\tau} \| w \| \partial |w'|^2 + e^{\frac{2\pi}{\tau} |w'|^2} \partial |w_n|,
$$
we have
\begin{equation*}
\begin{aligned}
\frac{\partial \| w \| \wedge \bar \partial \| w \| }{ \|w \|^2 (-\log \| w \|)^2} &= \frac{ \big(\frac{2\pi}{\tau} \big)^2  \partial |w'|^2 \wedge \bar \partial |w'|^2}{(-\log \|w \|)^2} + \frac{\frac{2\pi}{\tau} e^{\frac{2 \pi}{\tau} |w'|^2 }}{ \|w \| (-\log \| w \|)^2}  \partial |w'|^2 \wedge \bar \partial |w_n| \\
&\quad + \frac{\frac{2\pi}{\tau} e^{\frac{2 \pi}{\tau} |w'|^2 }}{\| w \| (-\log \| w \|)^2} \partial |w_n | \wedge \bar \partial |w'|^2  + \frac{e^{\frac{4\pi}{\tau} |w'|^2 }}{\| w \|^2 (-\log \| w \|)^2} \partial |w_n| \wedge \bar \partial |w_n|\\
&= \frac{ \big(\frac{2\pi}{\tau} \big)^2  \partial |w'|^2 \wedge \bar \partial |w'|^2}{(-\log \|w \|)^2} + \frac{1}{2} \sum_{k=1}^{n-1} \frac{\frac{2\pi}{\tau} \bar w_k w_n e^{\frac{2 \pi}{\tau} |w'|^2 }}{ |w_n| \|w \| (-\log \| w \|)^2}  dw_k \wedge d \bar w_n \, \\
&\quad+ \frac{1}{2} \sum_{k=1}^{n-1}  \frac{\frac{2\pi}{\tau}  w_k \bar w_n e^{\frac{2 \pi}{\tau} |w'|^2 }}{|w_n| \| w \| (-\log \| w \|)^2} d w_n \wedge d \bar w_k + \frac{e^{\frac{4\pi}{\tau} |w'|^2 }}{\| w \|^2 (-\log \| w \|)^2} \partial |w_n| \wedge \bar \partial |w_n|.
\end{aligned}
\end{equation*}
Therefore,
\begin{equation*}
\begin{aligned}
\omega_{\Omega_b^{(N)}} &= \sum_{k=1}^{n-1} \left(\frac{2\pi}{\tau} \frac{1}{(-\log \| w \|)} + \left(\frac{2\pi}{\tau} \right)^2 \frac{ { |w_k|^2 } }{(-\log \|w \|)^2} \right) dw_k \wedge d \bar w_k\\
&\quad+ {\frac{\left(\frac{2\pi}{\tau}\right)^2}{(-\log\|w\|)^2}}{ \sum_{1\leq k\neq j\leq n-1} \bar w_k w_j dw_k \wedge d \bar w_j } \\
&\quad + \frac{1}{2} \sum_{k=1}^{n-1} \frac{\frac{2\pi}{\tau} \bar w_k w_n e^{\frac{2 \pi}{\tau} |w'|^2 }}{ |w_n| \|w \| (-\log \| w \|)^2}  dw_k \wedge d \bar w_n  +  \frac{1}{2} \sum_{k=1}^{n-1}  \frac{\frac{2\pi}{\tau}  w_k \bar w_n e^{\frac{2 \pi}{\tau} |w'|^2 }}{|w_n| \| w \| (-\log \| w \|)^2} d w_n \wedge d \bar w_k \\
&\quad + \frac{1}{4} \frac{e^{\frac{4\pi}{\tau} |w'|^2 }}{ \| w \|^2 (-\log \| w \|)^2} d w_n \wedge d \bar w_n.
\end{aligned}
\end{equation*}

Now, we consider a frame   
\begin{equation*}
d \zeta_k := \frac{\sqrt{\frac{2 \pi}{\tau}}  d \omega_k}{\sqrt{-\log \| w \|}} \quad \text{for $k=1,\cdots, n-1$},  \quad d\zeta_n :=  \frac{1}{2} \frac{  e^{\frac{2\pi}{\tau} |w'|^2 } d w_n }{  \|w \| (-\log \|w \|)}
\end{equation*}
for $T_{\Sigma}^*$ on $\mathbb{D}^{n-1}(\epsilon) \times \mathbb{D}^* (\epsilon)$.
Then
\begin{equation*}
\begin{aligned}
\omega_{\Omega_b^{(N)}}  = &\sum_{k=1}^{n-1} \big( 1 + O({ | w_k |^2 } (-\log \| w \|)^{-1})  \big) d\zeta_k \wedge d \bar \zeta_k +  d\zeta_n \wedge d \bar \zeta_n \\
&+ {\sum_{1\leq k\neq j<n-1} O (\bar w_k w_j (-\log \|w \|)^{-1} ) d \zeta_k \wedge d \bar \zeta_j    }\\
&+ \sum_{k=1}^{n-1} O({\overline w_k}(-\log \| w \|)^{-\frac{1}{2} })  (d\zeta_k \wedge d \bar \zeta_n + d \zeta_n \wedge d \bar \zeta_k).
\end{aligned}
\end{equation*}
Since $ (-\log \|w \|)^{-1}$ and $ (-\log \|w \|)^{-\frac{1}{2} }$ converge to the zero as $\| w \| \sim |w_n| \rightarrow 0$, the lemma is proved.
\end{proof}

From now on, for simplicity, we write $\widetilde \omega_{\Omega_b^{(N)}}$ as $\widetilde \omega$. 

\begin{lemma}\label{curvature estimate 1} 
Let $\Gamma$ be a torsion-free lattice of $\text{Aut}(\mathbb{B}^n)$ with only unipotent parabolic automorphisms and let $\Sigma= \mathbb{B}^n/\Gamma$ be a ball quotient with finite volume. Let $b$ be a cusp of $\Gamma$. Let $(E,h^E)$ be a holomorphic vector bundle on $\overline \Sigma$ with a smooth hermitian metric $h^E$ on $\Sigma$. Then, for every point $p \in D_{b}$, there exists an open neighborhood $\mathbb{D}^n (\epsilon)$ near $p$ such that $p=(0,\cdots,0) \in \mathbb{D}^n (\epsilon)  \subset \subset \Omega_{b}^{(N)} $ and the inequality 
\begin{equation}\nonumber
\langle [ \sqrt{-1} \partial \bar \partial |w'|^2, \Lambda_{\widetilde \omega} ] \zeta, \zeta \rangle_{h^E, \, \widetilde \omega} \geq  (-\log \| w \| ) (\ell-1) |\zeta |^2_{h^E, \, \widetilde \omega}
\end{equation}
is satisfied for every point $q \in \mathbb{D}^{n-1} (\epsilon) \times \mathbb{D}^* (\epsilon)$ and for any $\zeta \in (\Lambda^{n,\ell} T_{\Sigma}^* \otimes E)_{q}$, where $w'=(w_1, \cdots, w_{n-1})$ are the euclidean coordinates on $\mathbb{D}^n(\epsilon)$. 
\end{lemma}

\begin{proof}
To prove the lemma, fix a point $p \in D_{b}$ and take a sufficiently small $\epsilon >0$ such that $p=(0,0,\cdots, 0) \in \mathbb{D}^{n}( \epsilon) \subset \subset \Omega_{b}^{(N)}$.  Take a point $q=(q_1, \cdots, q_n) \in \mathbb{D}^{n-1} (\epsilon) \times \mathbb{D}^* (\epsilon)$, and take a holomorphic frame $\{e_{\alpha} \}$ of $E$ near $q$.

Now, we define a holomorphic frame $\{d \tau_k \}$ for $T_{\Sigma}^*$ on $\mathbb{D}^{n-1}(\epsilon) \times \mathbb{D}^* (\epsilon)$ by 
$$
d \tau_k := \frac{d w_k}{\sqrt{(-\log \| q \|)}}, \quad k=1,\cdots,n-1 \quad \text{and} \quad  d \tau_n := \frac{d w_n}{ \| q \| (-\log \| q\|)}
$$
where $ \| q \| = e^{\frac{2\pi}{\tau} \sum_{k=1}^{n-1} |q_k |^2} |q_n|$. Under the frame, one can write $\widetilde \omega$ at $q$ as
\begin{equation}\label{simplified 1}
\begin{aligned}
\widetilde \omega_q &= \sum_{k=1}^{n-1} \frac{dw_k \wedge d \overline w_k}{(-\log \| q \|)} + \frac{dw_n \wedge d \overline w_n}{\|q \|^2 (-\log \| q \|)^2} \\
&= \sum_{k=1}^{n-1} d \tau_k \wedge d \overline \tau_k + d \tau_n \wedge d \overline \tau_n.
\end{aligned}
\end{equation}
Therefore,
\begin{equation}\label{simplified 2}
(i \partial \bar \partial |w'|^2)_{q} = \sqrt{-1}\sum_{k=1}^{n-1} dw_k \wedge d \bar w_k = \sqrt{-1}\sum_{k=1}^{n-1} (- \log \| q \| ) d \tau_k \wedge d \bar \tau_k.
\end{equation}
Since we can write 
$$
\zeta = \sum_{|I|=n,|J|=\ell, \alpha=1}^{n} u_{IJ\alpha} \, d\tau_{I} \wedge d \overline \tau_{J} \otimes e_{\alpha} \in (\Lambda^{n,\ell} T_{\Sigma}^* \otimes E)_{q},
$$
\eqref{simplified 1} and \eqref{simplified 2} imply that 
$$
\langle [ \sqrt{-1} \partial \bar \partial |w'|^2, \Lambda_{\widetilde \omega}] \zeta, \zeta \rangle_{h^E , \widetilde \omega} \geq (-\log \| w \|) (\ell - 1) | \zeta |^2_{h^E , \widetilde \omega}
$$
at $q$. Since $q$ can be chosen arbitrarily in $\mathbb{D}^{n-1}(\epsilon) \times \mathbb{D}^*(\epsilon)$, the proof is completed.
\end{proof}

\begin{lemma}\label{comparison between norms}

Let $\Gamma$ be a torsion-free lattice of $\text{Aut}(\mathbb{B}^n)$ with only unipotent parabolic automorphisms and let $\Sigma= \mathbb{B}^n/\Gamma$ be a ball quotient with finite volume. Let $b$ be a cusp of $\Gamma$.
For each point $p \in D_b$, let $\mathbb{D}^n (\epsilon)$ be an open neighborhood of $p$ in Lemma~\ref{quasi-isometric} so that there exist two positive constants $C_1$ and $C_2$ satisfying 
$$
C_1 \widetilde \omega \leq \omega \leq C_2 \widetilde \omega
$$
on $\mathbb{D}^{n-1}(\epsilon) \times \mathbb{D}^*(\epsilon)$.
Let $(E,h^E)$ be a holomorphic vector bundle on $\overline \Sigma$ with a smooth hermitian metric $h^E$ on $\Sigma$.

Then for each point $q \in \mathbb D^{n-1}(\epsilon)\times\mathbb D^*(\epsilon)$ and a holomorphic coordinate system $(z_1, \cdots, z_n)$ near $q$ satisfying
$$
\omega = \sum_{j=1}^{n} dz_j \wedge d \bar z_j, \quad \widetilde \omega = \sum_{j=1}^{n} \gamma_j dz_j \wedge d \bar z_j \quad \text{at} ~q
$$
with $\gamma_j>0$ for all $j$ and an orthonormal frame $\{e_\lambda \}$ of $E$ near $q$, the following inequalities hold:
$$
\frac{1}{C_2^n} \langle [\sqrt{-1} \Theta(E), \Lambda_{\omega}] S_{\gamma} u, S_{\gamma} u \rangle_{h^E, \omega}
\leq  \langle [\sqrt{-1} \Theta(E), \Lambda_{\widetilde \omega}] u ,  u \rangle_{h^E, \widetilde \omega} \leq \frac{1}{C_1^n} \langle [\sqrt{-1} \Theta(E), \Lambda_{\omega}]  S_{\gamma} u, S_{\gamma} u \rangle_{h^E, \omega}
$$
$$
C_1^{2( n +\ell )} |u|^2_{h^E, \omega} \leq |S_{\gamma} u |^2_{h^E, \omega} \leq C_2^{2( n +\ell )} |u|^2_{h^E, \omega} 
$$
for any 
$$
u = \sum_{|K|=\ell, \lambda} u_{K \lambda} dz_1 \wedge \cdots \wedge dz_n \wedge d \bar z_K \otimes e_\lambda \in (\Lambda^{n,\ell} T_{\Sigma}^* \otimes E)_q
$$
where 
$$
S_{\gamma} u := \sum_{|K|=\ell, \lambda} {(\gamma_1 \cdots \gamma_n \gamma_{K})}^{-1} u_{K \lambda} dz_1 \wedge \cdots \wedge dz_n \wedge d \bar z_K \otimes e_{\lambda}
$$
which is an invertible linear endomorphism on $(\Lambda^{n,\ell} T_{\Sigma}^*  \otimes E)_q$, and $
\gamma_K = \prod_{j \in K} \gamma_j$.
\end{lemma}

\begin{proof}
First, we note that 
$$
\gamma_j \geq \frac{1}{C_2}
$$
for every $j$ at the point $q$. Then
\begin{equation*}
|u|^2_{h^E, \widetilde \omega} =  \sum_{|K|=\ell,\lambda}  \frac{1}{\gamma_1 \cdots \gamma_n} \frac{1}{\gamma_K} |u_{K\lambda}|^2 \leq C_2^{n+\ell} \sum_{|K|=\ell, \lambda} |u_{K \lambda}|^2 = C_2^{n+\ell} |u|^2_{h^E, \omega}.
\end{equation*}
Moreover, by VIII-6 and VII-7 in \cite{Dem}
\begin{equation*}
\begin{aligned}
\Lambda_{\widetilde \omega} u = \sum_{|I|=\ell-1} \sum_{j,\lambda} \sqrt{-1} (-1)^{n+j-1} \gamma_{j}^{-1} u_{jI, \lambda} dz_1 \wedge \cdots \wedge \widehat{ dz_j} \wedge \cdots \wedge dz_n \wedge d \bar z_I \otimes e_{\lambda}
\end{aligned}
\end{equation*}
and
\begin{equation*}
\begin{aligned}
[\sqrt{-1}\Theta(E), \Lambda_{\widetilde \omega}]u  = \sum_{|I|=\ell-1} \sum_{j,k,\lambda,\mu} \gamma_j^{-1} c_{jk\lambda \mu} u_{jI, \lambda} dz_1 \wedge \cdots \wedge d z_n \wedge d \bar z_{kI} \otimes e_{\mu},
\end{aligned}
\end{equation*}
where 
$$
\sqrt{-1} \Theta(E) = \sqrt{-1}  \sum_{j,k,\lambda,\mu} c_{jk \lambda \mu} dz_j \wedge d \bar z_k \otimes e_{\lambda}^* \otimes e_{\mu}.
$$
Hence we obtain
\begin{equation*}
\begin{aligned}
\langle [\sqrt{-1}\Theta(E), \Lambda_{\widetilde \omega}]u , u \rangle_{h^E, \widetilde \omega} &= (\gamma_1 \cdots \gamma_n)^{-1} \sum_{|I|=\ell-1} \gamma_{I}^{-1} \sum_{j,k,\lambda,\mu} \gamma_j^{-1} \gamma_{k}^{-1} c_{jk\lambda \mu} u_{jI, \lambda} \overline {u_{kI, \mu}} \\
&= \gamma_1 \cdots \gamma_n\langle [\sqrt{-1} \Theta(E), \Lambda_{\omega}] S_{\gamma} u, S_{\gamma} u \rangle_{h^E, \omega} \\
& \geq \frac{1}{C_2^n} \langle [\sqrt{-1} \Theta(E), \Lambda_{\omega}] S_{\gamma} u, S_{\gamma} u \rangle_{h^E, \omega},
\end{aligned} 
\end{equation*}
and
\begin{equation}\nonumber
\begin{aligned}
    | S_{\gamma} u |^2_{h^E, \omega} 
    &= (\gamma_1 \cdots \gamma_n)^{-2} \sum_{|K|=\ell,\lambda} \gamma_K^{-2} |u_{K \lambda}|^2\\ &\leq C_2^{2( n + \ell)} \sum_{|K|=\ell, \lambda} |u_{K \lambda}|^2 = C^{2(n+\ell)}_{2} |u|^2_{h^E, \omega}.
\end{aligned}
\end{equation}
Furthermore, by a similar calculation, it follows 
$$
|u|^2_{h^E, \widetilde \omega} \geq C_1^{n+\ell} |u|^2_{h^E, \omega},
$$
$$
\langle [\sqrt{-1} \Theta(E), \Lambda_{\widetilde \omega} u, u \rangle_{h^E, \widetilde \omega} \leq \frac{1}{C_1^{n}} \langle [\sqrt{-1} \Theta(E), \Lambda_{\omega}] S_{\gamma}u, S_{\gamma} u \rangle_{h^E, \omega},
$$
and
$$
 |S_{\gamma} u|^2_{h^E, \omega} \geq C_1^{2(n+\ell)} | u |^2_{h^E, \omega}.
$$
Therefore, the proof is completed.
\end{proof}

\section{Proof of Theorem~\ref{main}}\label{proof of main theorem}

\subsection{Equivalence between certain cohomologies}
The proofs in this section are influenced by \cite{Fu92}, \cite{Z15}. Let $(F, h^F)$ be a holomorphic vector bundle over $\overline{\Sigma}$ which has a smooth hermitian metric $h^F$ on $\Sigma$. Given any open set $U \subset \overline{\Sigma}$, let $L^{r,*}_{2,loc}(U,F)$ be the space of measurable sections of $F$-valued $(r,*)$ forms which is $L^2$-integrable on $K- \bigcup_{b \in B} D_{b} $ for any compact subset $K$ of $U$ with respect to $dV_{\omega}$ and $h^F$. 
Define sheaves $L^{r,*}_{2,F}$ by
$$
L^{r,*}_{2,F} (U) := \{ \nu \in L^{r,*}_{2,loc} (U, F) : \bar \partial \nu\in L^{r, *+1}_{2, loc} (U, F) \}
$$
for any open set $U$ on $\overline{\Sigma}$.  Note that sheaves $L^{r,*}_{2,F}$ are fine sheaves by the completeness of $\omega$ on $\Sigma$ which has finite volume on $\Sigma$, and so
\begin{equation}\label{vanishing of fine sheaves}
H^{s} (\overline \Sigma, L^{r,*}_{2,F})=0 \quad \text{for every $s>0$}.
\end{equation}
For the detail of fine sheaves, see \cite{Warner}.

Let $\bar \partial_{(r,s),F}$ be the sheaf morphisms
$$
\bar \partial_{(r,s),F} : L^{r,s}_{2,F} \longrightarrow L^{r,s+1}_{2,F}
$$
which is defined by the maximal closed extension of $\bar \partial$. While the following lemma is likely well known to experts in this field, we include the proof to provide the reader with clarity.
\begin{lemma}\label{sheaf cohomology}
Let $\Gamma$ be a torsion-free lattice of $\text{Aut}(\mathbb{B}^n)$ and let $\Sigma = \mathbb{B}^n/ \Gamma$ be a ball quotient with finite volume. Let $E$ be a holomorphic vector bundle over $\overline{\Sigma}$ and $(F, h^F)$ be a holomorphic vector bundle $\overline{\Sigma}$ with a smooth hermitian metric $h^F$ on $\Sigma$. If there exist exact sequences 
\begin{equation}\label{exactness1}
0 \rightarrow \mathcal{O}(E) \xrightarrow{} L^{r,0}_{2,F} \xrightarrow {\bar \partial} \ker \bar  \partial_{(r,1),F} \rightarrow 0
\end{equation}
and
\begin{equation}\label{exactness2}
0 \rightarrow \ker \bar\partial_{(r,s),F}  \xrightarrow{} L^{r,s}_{2,F} \xrightarrow {\bar \partial} \ker \bar \partial_{(r,s+1),F} \rightarrow 0
\end{equation}
for all $1 \leq s \leq n$, then we have
$$
H^{s}(\overline{\Sigma}, \mathcal{O}(E)) \cong H^{r,s}_{L^2, \bar \partial} (\Sigma, F).
$$
\end{lemma}

\begin{proof}
If $s=0$, it follows by the exactness of \eqref{exactness1}.  If $s=1$, by \eqref{vanishing of fine sheaves} and the exactness of \eqref{exactness1}, we have an exact sequence
$$
0 
\to 
H^{0} ( \overline{\Sigma}, \mathcal{O} (E) ) 
\to 
H^{0}( \overline{\Sigma}, L^{r,0}_{2, F} )  
\to 
H^0(\overline \Sigma, \ker \bar \partial_{(r,1),F})
\to
H^{1}(\overline{\Sigma}, \mathcal{O}(E)) \to 0 
$$
and this implies
$$
H^1(\overline \Sigma, \mathcal{O}(E) ) \cong \frac{ H^0(\overline{\Sigma} , \ker \bar \partial_{(r,1),F})}{ \bar \partial H^0 (\overline {\Sigma}, L^{r,0}_{2,F}) }  \cong H^{r,1}_{L^2, \bar \partial} (\Sigma, F).
$$
Now, we assume that $s >1$.  Then by \eqref{vanishing of fine sheaves} and the exactness of \eqref{exactness1}, we have an exact sequence
\begin{equation}\nonumber
0\cong H^{s-1}( \overline{\Sigma}, L^{r,0}_{2,F} )
\to
H^{s-1}( \overline{\Sigma}, \ker \bar \partial_{(r,1),F} ) 
\to
H^s(\overline{\Sigma}, \mathcal{O}( E)) 
\to 
H^s (\overline{\Sigma}, L^{r,0}_{2,F}) \cong 0 
\end{equation}
for any $s>1$. Hence
\begin{equation}\label{cohomology 1}
H^s(\overline \Sigma, \mathcal{O}( E) ) \cong H^{s-1} (\overline \Sigma, \ker \bar \partial_{(r,1),F}).
\end{equation}
Moreover, by \eqref{vanishing of fine sheaves} and the exactness of \eqref{exactness2} we have an exact sequence
\begin{equation}\nonumber
0\cong H^{s-1}( \overline{\Sigma}, L^{r,*}_{2,F} ) 
\to
H^{s-1}( \overline{\Sigma}, \ker \bar \partial_{(r,*+1),F} ) 
\to
H^{s}(\overline{\Sigma}, \ker \bar \partial_{(r,*),F}) \to
H^s (\overline{\Sigma}, L^{r,*}_{2,F}) \cong 0 
\end{equation}
for any $s>1$. Hence
\begin{equation}\label{cohomology 2}
H^{s-1}(\overline{\Sigma}, \ker \bar \partial_{(r,*+1),F}) \cong H^s(\overline{\Sigma}, \ker \bar \partial_{(r, *),F}).
\end{equation}
Therefore,
\begin{equation}\nonumber
\begin{aligned}
H^{s}(\overline{\Sigma}, \mathcal{O}(E) ) & \cong H^{s-1} (\overline \Sigma, \ker \bar \partial_{(r,1),F}) \\
& \cong H^{s-2} (\overline {\Sigma}, \ker \bar \partial_{(r,2),F} ) \\
& \cong \cdots \\
& \cong H^{1} ( \overline {\Sigma}, \ker \bar \partial_{(r,s-1),F} ),
\end{aligned}
\end{equation}
by \eqref{cohomology 1} and \eqref{cohomology 2} for any $s>1$.
Finally, by \eqref{vanishing of fine sheaves} and the exactness of \eqref{exactness2} we have an exact sequence
$$
0 
\to 
H^{0}(\overline \Sigma, \ker \bar \partial_{(r,s-1),F}) 
\to 
H^{0}(\overline \Sigma, L^{r,s-1}_{2,F}) 
\to 
H^{0}(\overline \Sigma, \ker \bar \partial_{(r,s),F}) 
\to 
H^1(\overline \Sigma, \ker \bar \partial_{(r,s-1),F}) \to  
0,
$$
and so 
$$
H^1(\overline{\Sigma}, \ker \bar \partial_{(r,s-1),F}) \cong \frac{H^0(\overline \Sigma, \ker \bar \partial_{(r,s),F})}{\bar \partial H^0 (\overline {\Sigma}, L^{r, s-1}_{2,F})} \cong H^{r,s}_{L^2, \bar \partial}(\Sigma, F).
$$
This completes the proof.
\end{proof}

\subsection{Existence of resolution of a sheaf when $r=0$ or $r=n$} Throughout the subsection, we let $F=(S^m T_{\overline \Sigma}^*, g^{-m})$, where $g$ is the induced Bergman metric on $\Sigma$.  Let $\mathbb{D}(\epsilon)$ denote the complex disc in $\mathbb{C}$ with radius $\epsilon$ and $\mathbb{D}^*(\epsilon) := \mathbb{D} (\epsilon) - \{0 \}$. Let $D_b$ be the boundary divisor in $\overline{\Sigma}$ corresponding to the cusp $b$.

From now on, we will use the euclidean coordinates $(w', w_n):= (w_1, \cdots, w_n)$ as local holomorphic coordinates on $\Omega_{b}^{(N)}$ which are induced by the uniformization $\widehat G^{(N)}$ of $\Omega_b^{(N)} = \widehat G^{(N)} / \pi(\Gamma \cap W_b)$ and we assume that $\| w \| < 1$ on $\Omega_{b}^{(N)}$. 

Let $b_1, \cdots, b_k$ be the cusps of $\Sigma$ and $D_j$ be the boundary divisor corresponding to $b_j$ for each $j=1, \cdots, k$ in the toroidal compactification $\overline{\Sigma}$ of $\Sigma$. 
Let $\pi_j : \Omega_{j}^{(N)} \rightarrow D_{j}$ be the canonical projection. 
Define
\begin{equation}\nonumber
E_{r,m}:=\bigoplus_{j=1}^k E_{r,m,j}
\end{equation}
where
\begin{equation}\nonumber
\begin{aligned}
E_{r,m,j} 
:=& \bigg \{ \Lambda^r \pi_j^* T_{D_j}^*  \otimes \bigg( \bigoplus_{\substack{  \ell < n-(m+r) }  } 
\left( S^{m-\ell} \pi_j^* T_{D_j}^* \otimes S^{\ell} T^*_{\Omega_j /D_j} \otimes [\ell D_j] \right)    \\
& \quad \quad  \oplus \bigoplus_{ \substack{ \ell \geq n-(m+r) } }  \left( { S^{m-\ell} \pi_{j}^* T_{D_j}^*  \otimes S^{\ell} T^*_{\Omega_j /D_j} } \otimes [(\ell-1) D_j] \right) \bigg)\bigg \}  \\
&\oplus \bigg \{ \left(  \Lambda^{r-1} \pi_j^* T_{D_j}^*  \otimes T_{\Omega_j / D_j}^* \right) \otimes \bigg(  \bigoplus_{\substack{ \ell < n-(m+r+1) }  } \left( S^{m-\ell} \pi_j^* T_{D_j}^* \otimes S^{\ell} T^*_{\Omega_j /D_j} \otimes [(\ell+1) D_j]   \right) \\
&\qquad \qquad \oplus   \bigoplus_{\substack{ \ell \geq n-(m+r+1) }  } \left( S^{m-\ell} \pi_j^* T_{D_j}^* \otimes S^{\ell} T^*_{\Omega_j /D_j} \otimes [\ell D_j]   \right)\bigg)
\bigg \}
\end{aligned}
\end{equation}
be a vector bundle over  $\bigcup_{j=1}^{k} \Omega_{i}^{(N)}$
and $E_{r,m}=\Lambda^{r} T^*_{\Sigma} \otimes S^mT^*_\Sigma$ on $\Sigma$, when $r \geq 0$. 
This is well defined since the restriction of the line bundle $[tD_j]$ to $\Sigma$ is trivial for each $j$ and $t \in \mathbb N\cup \{0\}$.
Here $[D_{j}]$ is the associated holomorphic line bundle to $D_{j}$ and $T^*_{\Omega_{j}^{(N)} / D_{j}}$ is the relative cotangent bundle for $\pi_j$. In the definition of $E_{r,m}$, we use the canonical identification of $\Omega_{j}^{(N)}$ and a tabular neighborhood of the zero section of $N_{j}$.

\begin{lemma}\label{exactness for (r,s) 0}
Let $\Gamma$ be a torsion-free lattice of $\text{Aut}(\mathbb{B}^n)$ with only unipotent parabolic automorphisms and let $\Sigma = \mathbb{B}^n/ \Gamma$ be a ball quotient with finite volume.  { Then for each $r, m \in \mathbb{N} \cup \{ 0 \}$,} the following sequence
\begin{equation*}
0 \rightarrow \mathcal{O} ( E_{r,m} ) \xrightarrow { } \ L^{r,0}_{2,F} \xrightarrow {\bar \partial} \ker \bar  \partial_{(r,1),F}
\end{equation*}
is exact.
\end{lemma}
\begin{proof}
At first we want to show that 
\begin{equation}\label{exactness for (s,q) 1}
\mathcal{O}_{\overline {\Sigma}} ( E_{r,m} ) \xrightarrow { } L^{r,0}_{2,F} \xrightarrow {\bar \partial} \ker \bar \partial_{(r,1),F}
\end{equation}
is exact.
For any open set in $\Sigma$, the exactness of the sequence is well known and hence we only need to consider open sets intersecting with divisors $D_{j}, j=1, \cdots, k$ in $\overline{\Sigma}$.
Consider a divisor $D_{b} \subset \overline{\Sigma}$ and fix a point $p \in D_{b}$. Let $U_{p} := \mathbb{D}^{n-1} (\epsilon) \times \mathbb{D}^* (\epsilon)$ where $\mathbb{D}^{n-1} (\epsilon) := \{ (w_1, \cdots, w_{n-1}) : w_i \in \mathbb{D}(\epsilon) , 1 \leq i \leq n-1 \}$ with a sufficiently small $\epsilon >0$ so that $p \in U_{p} \subset \subset \Omega_{{b}}^{(N)}$. 

Note that every holomorphic section $s$ in $L^{r,0}_{2,F}(U_{p})$ is  expressed as 
$$
s= \sum_{\substack{|I|=m \\ j_1 < \cdots < j_r}} s_{j_1 \cdots j_r, I} (w', w_n) dw_{j_1} \wedge \cdots \wedge dw_{j_r} \otimes  dw^{I}
\quad \text{ with } \quad s_{j_1 \cdots j_r, I} (w', w_n) \in \mathcal{O}(U_{p}).
$$
Since for each fixed $w'$, $s_{j_1 \cdots j_r, I}$ can be regarded as a holomorphic function on $\mathbb{D}^* (\epsilon)$, $s_{j_1 \cdots j_r, I} (w', w_n)$ has the Laurent expansion:
$$
s_{j_1 \cdots j_r, I} (w', w_n) = \sum_{k=-\infty}^{\infty} s_{j_1\cdots j_r, I, k}(w',0) w_n^{k}.
$$
By Lemma \ref{quasi-isometric}, 
\begin{equation*}
\| dw_k \|^2_{\omega^{-1}} \sim    
\left \{ \begin{array}{ll} 
-\log \|w \| & \text{if $k=1, \cdots, n-1$},\\
\|w\|^2 (-\log\|w\|)^{2} & \text{if $k=n$}.
\end{array} \right.
\end{equation*} 
Let $dV$ be the Lebesgue measure. Then, $s \in L^{r,0}_{2,F}(U_{p})$ implies that
\begin{equation}\nonumber
\int_{\mathbb D^{n-1}(\epsilon)\times\mathbb{D}^*(\epsilon)} |s_{j_1 \cdots j_r, I} (w', w_n)|^{2} \| w \|^{2(i_n-1)}(-\log \| w \|)^{r+m+i_n - (n+1)} dV < \infty \quad \text{for any $|I|=m$}
\end{equation}
when $n \not = j_r$ and
\begin{equation}\nonumber
\int_{\mathbb D^{n-1}(\epsilon)\times\mathbb{D}^*(\epsilon)} |s_{j_1 \cdots j_r, I} (w', w_n)|^{2} \| w \|^{2 i_n }(-\log \| w \|)^{{(r+1)}+m+i_n - (n+1)} dV < \infty \quad \text{for any $|I|=m$}
\end{equation}
when $n = j_r$. Let $dV_{w_n}$ be the Lebesgue measure of $\mathbb{C}$ for each fixed $w' \in \mathbb{D}^{n-1}(\epsilon)$. Then,  since for any $k$ the function $x\mapsto x^2(-\log x)^k$ is increasing on $0<x<\epsilon$ for sufficiently small $\epsilon$, we have 
\begin{equation}\label{integral calculation for (r,0) 1}
\int_{\mathbb{D}^*(\epsilon)} |s_{j_1 \cdots j_r, I} (w', w_n)|^{2} | w_n|^{2(i_n-1)}(-\log | w_n|)^{r+m+i_n - (n+1)} dV_{w_n} < \infty \quad \text{for any $|I|=m$}
\end{equation}
when $n \not = j_r$ and
\begin{equation}\label{integral calculation for (r,0) 2}
\int_{\mathbb{D}^*(\epsilon)} |s_{j_1 \cdots j_r, I} (w', w_n)|^{2} | w_n|^{2 i_n }(-\log | w_n|)^{r+m+i_n -n} dV_{w_n} < \infty \quad \text{for any $|I|=m$}
\end{equation}
when $n = j_r$.
Using polar coordinate $(\rho, \theta)$,  then we obtain
\begin{equation}\nonumber
\begin{aligned}
\eqref{integral calculation for (r,0) 1} &= \int_{0}^{\epsilon} \bigg( \int_{0}^{2\pi} \bigg | \sum_{k=-\infty}^{\infty} s_{j_1 \cdots j_r, I, k}(w',0) (\rho^{k} e^{\sqrt{-1} k \theta})   \bigg |^2 d \theta \bigg)  \rho^{2(i_k-1)} (-\log \rho)^{r+m+i_n - (n+1)} \rho d\rho  \\
&= 2\pi  \sum_{k=-\infty}^{\infty} |s_{j_1 \cdots j_r, I,k} (w',0)|^2 \bigg( \int_{0}^{\epsilon}  \rho^{2k + 2i_n -1} (- \log \rho)^{r+m+i_n - (n+1)} d\rho \bigg) \\
&= 2\pi \sum_{k=-\infty}^{\infty} |s_{j_1 \cdots j_r, I, k} (w', 0)|^2 \bigg( \int_{-\log \epsilon}^{\infty} e^{-(2k+2i_n) u} u^{r+m+i_n-(n+1)}du \bigg)
\end{aligned}
\end{equation}
and similarly, we obtain 
\begin{equation}\nonumber
\begin{aligned}
\eqref{integral calculation for (r,0) 2} &= 2\pi \sum_{k=-\infty}^{\infty} |s_{j_1 \cdots j_r, I, k} (w', 0)|^2 \bigg( \int_{-\log \epsilon}^{\infty} e^{-(2k+2i_n+2) u} u^{r+m+i_n-n}du \bigg).
\end{aligned}
\end{equation}
Remark that 
\begin{equation} \label{integral convergence for (r,0) 1}
\int_{-\log \epsilon}^{\infty} e^{-(2k+2i_n) u} u^{r+m+i_n-(n+1)}du  < \infty \Longleftrightarrow
\bigg \{ \begin{array}{ll} 
k \geq -(i_n -1) & \text{if $r+m+i_n \geq n$},\\
k \geq -i_n & \text{if $r+m+i_n < n$}
\end{array} 
\end{equation}
and
\begin{equation} \label{integral convergence for (r,0) 2}
 \int_{-\log \epsilon}^{\infty} e^{-(2k+2i_n+2) u} u^{(r+1)+m+i_n-(n+1)}du  < \infty \Longleftrightarrow
\bigg \{ \begin{array}{ll} 
k \geq -i_n & \text{if $r+m+i_n \geq n-1$},\\
k \geq -(i_n+1) & \text{if $r+m+i_n < n-1$}.
\end{array} 
\end{equation} 
Therefore, \eqref{exactness for (s,q) 1} is exact.
Secondly, we want to show that 
\begin{equation*}
0 \rightarrow \mathcal{O} ({ E_{r,m}}) \xrightarrow {} L^{r,0}_{2,F} 
\end{equation*}
is exact, i.e. we will prove that if $s \in \mathcal{O} ( E_{r,m}) (U_{p})$, then it follows that  $s \in L^{r,0}_{2,F}(U_{p})$. 
Note that  by the definition of ${E_{r,m}}$, we can express $s$ as
$$
s= \sum_{\substack{|I|=m \\ j_1 < \cdots <j_r } } s_{j_1 \cdots j_r, I} (w', w_n) dw_{j_1} \wedge \cdots \wedge dw_{j_r} \otimes dw^{I}
\quad \text{ with } \quad s_{j_1 \cdots j_r, I} (w', w_n) \in \mathcal{O}(U_{p}).
$$
If $n \not = j_r$, then $s_{j_1 \cdots j_r, I}$ may have a pole up to the order { $i_n -1$ if $i_n \geq n-(m+r)$ and $i_n$ if $ i_n < n-(m+r)$, for each fixed $w'$.} Similarly, if $n = j_r $ then $s_{j_1 \cdots j_r, I}$ may have a pole up to the order { $i_n $ if $i_n \geq n-(m+r+1)$ and $i_r +1 $ if $ i_r < n-(m+r+1)$, for each fixed $w'$.} Finally, using \eqref{integral convergence for (r,0) 1} and \eqref{integral convergence for (r,0) 2}, we know that $s \in L_{2,loc}^{(r,0),F}(U_{p})$ and the proof is completed.
\end{proof}

The following proposition can be regarded as an $L^2$-version of Dolbeault–Grothendieck lemma.
\begin{proposition}\label{exactness for (r,s) 2}
Let $\Gamma$ be a torsion-free lattice of $\text{Aut}(\mathbb{B}^n)$ with only unipotent parabolic automorphisms and $\Sigma = \mathbb{B}^n/ \Gamma$ be a ball quotient with finite volume. Let $F=S^mT_{\Sigma}^*$. Then the following sequences
\begin{equation}\label{local dolbeault for (r,s)}
0 \rightarrow \ker \bar\partial_{(r,\ell),F}  \xrightarrow {} L^{(r,\ell)}_{2,F} \xrightarrow {\bar \partial} \ker \bar \partial_{(r,\ell+1),F} \rightarrow 0
\end{equation}
are exact for all $1 \leq \ell \leq n$. If $r=0$ or $n$, then \eqref{local dolbeault for (r,s)} is exact for $\ell =0$.
\end{proposition}

\begin{proof}
For any open set in $\Sigma$, the exactness of the sequence is well known and hence we only need to consider open sets intersecting with the corresponding divisor $D_{b}$ in $\overline{\Sigma}$ of a cusp $b$ of $\Gamma$. Consider a divisor $D_{b} \subset \overline \Sigma$ and fix a point $p \in D_{b}$. Take a sufficiently small open polydisc $p \in U_p \subset \subset \Omega_{b}^{(N)}$ of $p$ such that there exists a coordinate system $(w_1,\cdots,w_n)$ on an open set $W_p$ containing $\overline{U_p}$ which satisfies the statements of Lemma~\ref{quasi-isometric}, \ref{curvature estimate 1}, and \ref{comparison between norms}. Hence, we can set $U_p = \{ w \in \Omega_{{b}}^{(N)} : |w_1|^2 < \delta, \cdots, |w_n|^2 < \delta \}$ for a sufficiently small real number $\delta>0$ and let $\omega_\epsilon$ be a complete K\"ahler metric defined by
\begin{equation*}
\omega_{\epsilon} = \omega + \epsilon \sqrt{-1}  \partial \bar \partial \bigg( \sum_{i=1}^{n} \frac{1}{\delta^2 - |w_i|^2} \bigg) \quad \text{on} \quad U_p-D_b.
\end{equation*}

Let $\sigma_r$ be the induced metric on $\Lambda^{(r,0)} T^*_{U_p - D_b} $ from $\omega$. Since the metric $\omega$ on $\Sigma$ is induced by the Bergman metric on $\mathbb{B}^n$, there is a non-zero constant $\alpha$ on $U_p - D_{b}$ such that
$$
 \langle [\sqrt{-1}\Theta(\Lambda^{(r,0)} T^*_{U_p - D_b} \otimes F \otimes K^{-1}_{U_p - D_b} ), \Lambda_{\omega}] \zeta, \zeta \rangle_{h^F \sigma_r, \, \omega}  \geq  \alpha | \zeta |^2_{h^F \sigma_r, \, \omega}
$$
for any  $\zeta \in C^{\infty} \big( U_p - D_b, \Lambda^{n,\ell+1} T_{\Sigma}^* \otimes \big( \Lambda^{(r,0)} T^*_{U_p - D_b} \otimes F \otimes K^{-1}_{U_p - D_b} \big) \big)$. 
Then, by Lemma~\ref{quasi-isometric}, \ref{curvature estimate 1}, and \ref{comparison between norms}, we obtain
\begin{equation}\label{omega_inequality_for (r,s)_1}
\begin{aligned}
& \langle [\sqrt{-1}\Theta(\Lambda^{(r,0)} T^*_{U_p - D_b} \otimes F \otimes K^{-1}_{U_p - D_b}), \Lambda_{\omega} ] \zeta, \zeta \rangle_{ h^F e^{-k|w'|^2} \sigma_r, \omega}\\
&= \langle [\sqrt{-1}\Theta(\Lambda^{(r,0)} T^*_{U_p - D_b} \otimes F \otimes K^{-1}_{U_p - D_b}), \Lambda_{ \omega} ]  \zeta, \zeta \rangle_{h^F e^{-k |w'|^2}\sigma_r , \omega} \\
&\quad+ k \langle [\sqrt{-1} \partial \bar \partial |w'|^2 ,  \Lambda_{ \omega} ]  \zeta,  \zeta \rangle_{h^F e^{-k|w'|^2} \sigma_r, \omega }  \\
&\gtrsim \langle [\sqrt{-1}\Theta(\Lambda^{(r,0)} T^*_{U_p - D_b} \otimes F \otimes K^{-1}_{U_p - D_b}), \Lambda_{\omega} ]  \zeta,  \zeta \rangle_{h^F e^{-k |w'|^2} \sigma_r, \omega} \\
&\quad + k \langle [\sqrt{-1} \partial \bar \partial |w'|^2 ,  \Lambda_{\widetilde \omega} ]  S_{\gamma}^{-1} \zeta, S_{\gamma}^{-1}  \zeta \rangle_{h^F e^{-k |w'|^2} \sigma_r, \widetilde \omega }\\
&\gtrsim \alpha | \zeta |^2_{h^F e^{-k |w'|^2} \sigma_r,  \omega} + k\ell  |S_{\gamma}^{-1} \zeta|^2_{h^F e^{-k|w'|^2} \sigma_r, \widetilde \omega}  \\
&\gtrsim \alpha | \zeta |^2_{h^F e^{-k|w'|^2} \sigma_r ,  \omega} + k \ell | \zeta |^2_{h^F e^{-k |w'|^2} \sigma_r, \omega} \\
&\gtrsim (\alpha+ k \ell ) | \zeta|^2_{h^F e^{-k |w'|^2} \sigma_r,  \omega}
\end{aligned}
\end{equation}
for any $\zeta \in \big( \Lambda^{n,\ell+1} T_{\Sigma}^* \otimes \big( \Lambda^{(r,0)} T^*_{U_p - D_b} \otimes F \otimes K^{-1}_{U_p - D_b}\big) \big)$. 
Since $\omega_{\epsilon} \geq \omega$, by the proof of Lemma~6.3 in Chapter VIII of \cite{Dem}, for any $\epsilon>0$,
$$
\langle [\sqrt{-1}\Theta(\Lambda^{(r,0)} T^*_{U_p - D_b} \otimes F \otimes K^{-1}_{U_p - D_b}), \Lambda_{\omega_{\epsilon}}] \zeta, \zeta \rangle_{h^F e^{-k|w'|^2} \sigma_r, \omega_{\epsilon}} >0
$$
in bidegree $(n,\ell+1)$ for sufficiently big $k$ and $\ell\geq 1$ due to \eqref{omega_inequality_for (r,s)_1}. 
Therefore, by the proof of Theorem 6.1 in Chapter VIII of \cite{Dem}, if $k>0$ is sufficiently large, it follows that   
\begin{equation} \label{vanishing for (r,s) 1}
\begin{aligned}
&H^{r,\ell+1}_{L^2,\bar \partial}(U_p - D_{b}, F, h^F e^{-k |w'|^2} , \omega) \\
&\cong H^{0,\ell+1}_{L^2, \bar \partial} (U_p - D_{b}, \Lambda^{(r,0)} T_{U_p - D_b}^* \otimes F,  h^F e^{-k|w'|^2} \sigma_r, \omega) \\
&\cong H^{n,\ell+1}_{L^2, \bar \partial}(U_p - D_{b}, \Lambda^{(r,0)} T_{U_p - D_b}^* \otimes F \otimes K^{-1}_{U_p - D_b}, h^F e^{-k |w'|^2} (\det \omega)^{-1} \sigma_r , \omega) = 0
\end{aligned}
\end{equation}
when $1 \leq \ell \leq n$. 

If $r=0$ or $n$, $\Lambda^{(r,0)} T_{\Sigma}^* \otimes F$ is Nakano positive by the Nakano positivity of $T_{\Sigma}^*$. Hence by Corollary \ref{basic estimate by Kahler potential 2}
\begin{equation}\label{vanishing for (r,s) 2}
H^{n,\ell+1}_{L^2, \bar \partial} (U_p - D_b, \Lambda^{(r,0)} T_{U_p- D_b}^* \otimes F, {h^F} \sigma_r, \omega_{1}) =0
\quad \text{ for any } 0 \leq \ell \leq n
\end{equation}
because $  -\log (-\log \| w \|) - \sum_{i=1}^{n} \big( \frac{1}{\delta^2 - |w_i|^2 } \big) $ is a K\"ahler potential of $\omega_{1}$ which has uniformly bounded norm for $\omega_1$. 

Now, we define a holomorphic frame $\{d \tau_i \}$ for $T_{\Sigma}^*$ on $U_p - D_p$ by 
$$
d \tau_i := \frac{d w_i}{\sqrt{(-\log \| q \|)}}, \quad i=1,\cdots,n-1 \quad \text{and} \quad  d \tau_n := \frac{d w_n}{ \| q \| (-\log \| q\|)}
$$
and define
$$
\lambda_i := 1+ (-\log \|w \|) \frac{\delta^2 + |w_i|^2}{(\delta^2 - |w_i|^2)^3}, \quad i=1,\cdots, n-1 
$$
and
$$
\lambda_n := \bigg( 1 + \| w\|^2 (-\log \|w \|)^2 \bigg) \frac{\delta^2 + |w_n|^2}{(\delta^2 - |w_n|^2)^3}.
$$
Then $\widetilde \omega$ in Lemma~\ref{quasi-isometric}  is equal to $ \sum_{i=1}^{n} d\tau_i \wedge d \bar \tau_i$ and $$
\sum_{i=1}^{n} \lambda_i d \tau_i \wedge d \bar \tau_i  = \widetilde \omega + \sqrt{-1} \partial \bar \partial \bigg( \sum_{i=1}^{n} \frac{1}{\delta^2 - |w_i|^2} \bigg) \sim \omega_1.
$$
Now, for any $(0,s)$ form $\alpha$ 
$$
|\alpha \wedge d \tau_1 \wedge \cdots \wedge d \tau_n |^2_{\omega_1} dV_{\omega_1}  \lesssim |\alpha|^2_{\omega} dV_{\omega} \quad \text{on} \quad U_p - D_b
$$
and for any $(0,0)$ form $u$,
$$
|u d\tau_1 \wedge \cdots \wedge d \tau_n |^2_{\omega_1} dV_{\omega_1} \approx |u|^2 dV_{\omega} \quad \text{on} \quad U_p - D_b
$$
by the following proof of Lemma 6.3 in Chapter VIII of \cite{Dem} with using $\omega_1 \sim \sum_{i=1}^{n} \lambda_i d \tau_i \wedge d \bar \tau_i$ and $\omega \sim \widetilde \omega$. Hence under the trivialization of $K_{\Sigma} |_{U_p - T_p}$ induced by $(\tau_1, \cdots, \tau_n)$ on $U_p - D_p$, it follows that $$
L^{0,0}_{2} (U_p-D_p, \Lambda^{(r,0)} T_{\Sigma}^* \otimes F, h^F \sigma_r, \omega) \cong L^{n,0}_{2} (U_p-D_p, \Lambda^{(r,0)} T_{\Sigma}^* \otimes F, h^F \sigma_r, \omega_1)
$$
and
$$
L^{0,1}_{2} (U_p-D_p, \Lambda^{(r,0)} T_{\Sigma}^* \otimes F, h^F \sigma_r, \omega) \hookrightarrow L^{n,1}_{2} (U_p-D_p, \Lambda^{(r,0)} T_{\Sigma}^* \otimes F, h^F \sigma_r, \omega_1).
$$
Therefore, by \eqref{vanishing for (r,s) 2} it follows that
\begin{equation}\label{vanishing for (r,1)}
H^{0,1}_{L^2, \bar \partial} (U_p - D_b, \Lambda^{(r,0)} T_{U_p- D_b}^* \otimes F, {h^F}\sigma_r, \omega) =0.
\end{equation}
Since $\pm |w'|^2$ is locally integrable for the Lebesgue measure, the sequence of Lemma is exact if we take a smaller open neighborhood $V_p \subset \subset U_p$ of $p$ by \eqref{vanishing for (r,s) 1} and \eqref{vanishing for (r,1)}. Therefore, the proof is completed.
\end{proof}

\subsection{An $L^2$-version of Dolbeault-Grothendieck lemma for $1 \leq r \leq n-1$ }

In this subsection, to control intermediate cases for $r$, i.e. for $1 \leq r \leq n-1$, $\ell=0$ in Proposition~\ref{exactness for (r,s) 2}, we prove the following lemma using the argument of Chen given in \cite{Ch11}.

\begin{lemma}\label{local L2 estimate}
Let $\Gamma$ be a torsion-free lattice of $\text{Aut}(\mathbb{B}^n)$ with only unipotent parabolic automorphisms and let $\Sigma= \mathbb{B}^n/\Gamma$ be a ball quotient with finite volume. Let $b$ be a cusp of $\Gamma$.  For each point $p \in D_{b}$, let $\mathbb{D}^n(2\epsilon)$ be an open neighborhood of $p$ given in Lemma \ref{quasi-isometric} satisfying $p \in \mathbb{D}^n(2\epsilon) \subset \subset \Omega_{b}^{(N)}$. Let $\tau$ be a smooth plurisubharmonic function on $\mathbb{D}^{n-1} (2\epsilon) \times \mathbb{D}^*(2\epsilon)$.
Then for each { $s \geq 0$ and $\ell \geq 0$}, there exists a positive constant $C$, which is independent of $\epsilon$, and exists a positive constant $ \epsilon ' \leq \epsilon $ such that for any $(0,1)$ form 
$$
\alpha = \sum_{j=1}^{n} \alpha_j d \bar w_j, \quad \bar \partial \alpha =0
$$ 
on $\mathbb{D}^{n-1}(2 \epsilon) \times \mathbb{D}^*(2 \epsilon) $ { satisfying} 
\begin{equation}\label{norm for (0,1)}
\begin{aligned}
\| \alpha \|^2_{s, \ell, \tau} 
&:= \sum_{j=1}^{n-1} \int_{\mathbb{D}^{n-1} (2\epsilon) \times \mathbb{D}^*(2\epsilon) } |\alpha_j|^2 {(-\log \|w \|)}^{s-1}  \|w\|^{2 (\ell-1)} e^{-\tau} dV \\
&\quad+ \int_{\mathbb{D}^{n-1}(2\epsilon) \times \mathbb{D}^* (2\epsilon) } |\alpha_n|^2 (-\log \|w \|)^{s} \| w \|^{2 \ell} e^{-\tau} dV < \infty,
\end{aligned}
\end{equation}
there exists a solution $u$ of the $\bar\partial$-equation $\bar \partial u = \alpha$ on $\mathbb{D}^{n-1}(2\epsilon) \times \mathbb{D}^* (\epsilon') $ and it satisfies
\begin{equation}\nonumber
\int_{\mathbb{D}^{n-1} (2\epsilon) \times \mathbb{D}^* (\epsilon') } |u|^2 {\| w \|^{2 (\ell-1)} (-\log \| w \|)^{s-2}} e^{-\tau} dV \leq C  \| \alpha \|^2_{s,\ell,\tau}
\end{equation}
where $ dV $ is the Lebesgue measure. 
\end{lemma}

\begin{proof}
Let $\chi: \mathbb{R} \rightarrow [0,1]$ be a smooth function satisfying $\chi |_{(-\infty, \frac{1}{4})}=1$ and $\chi |_{(\frac{3}{4},\infty)}=0$.  Let $ \delta \leq \epsilon$ be a small positive constant such that 
\begin{equation}\nonumber
|w_n|^2 + \delta^2 <e^{-e}
\end{equation} 
on $\mathbb{D}^n(2\epsilon)$ and let 
\begin{enumerate}
\item $\varphi$ : a smooth plurisubharmonic function on { $\mathbb{D}^{n-1}(2 \epsilon) \times \mathbb{D}^* (2 \epsilon)$ }
\item $\rho = \log (| w_n|^2 + \delta^2)$
\item $\eta = - \rho + \log (- \rho)$
\item $\psi = - \log \eta$
\item $\phi = |w'|^2 + \varphi + \log (| w _n|^2 + \delta^2)$.
\end{enumerate}
Since
\begin{equation}\label{complex hessian}
\partial \bar \partial \psi =  - \frac{\partial \bar \partial \eta}{\eta} + \frac{\partial \eta \wedge \bar \partial \eta}{\eta^2} =(1 + (-\rho)^{-1}) \frac{\partial \bar \partial \rho}{\eta} + \frac{\partial \rho \wedge \bar \partial \rho}{\eta \rho^2} + \frac{\partial \eta \wedge \bar \partial \eta}{\eta^2}, 
\end{equation}
and
$$
\partial \bar \partial \rho = \frac{\delta^2 dw_n \wedge d \bar w_n} {(|w_n|^2 + \delta^2)^2},
$$
$\psi$ is a smooth plurisubharmonic function.

Let 
$$
v := \alpha \chi \bigg( \frac{| w_n |^2}{ \delta ^2}   \bigg).
$$
Since $\alpha$ is $L^2$-integrable on $\mathbb{D}^n (2 \epsilon)$ for $| w_n |^{2\ell} dV$ by the finiteness of \eqref{norm for (0,1)} and $\|w \| \sim |w_n|$, we obtain 
$$
v \in L^2(\mathbb{D}^{n-1} (2\epsilon) \times \mathbb{D}^*(\epsilon), |w_n|^{2\ell}),
$$ 
where $dV$ is the Lebesgue measure. Hence, by the pseudoconvexity of $\mathbb{D}^{n-1}(2\epsilon) \times \mathbb{D}^* (\epsilon)$ and pluriharmonicity of $\log |w_n|^2$, the minimal solution $u_{\delta}$ of $\bar \partial u_{\delta} = v$ exists in $L^2(\mathbb{D}^{n-1} (2\epsilon) \times \mathbb{D}^* (\epsilon), e^{-\phi} |w_n|^{2\ell})$. Moreover, since $\psi$ is a bounded function, $u_{\delta} e^{\psi}$ is orthogonal to $L^2(\mathbb{D}^{n-1} (2\epsilon) \times \mathbb{D}^*(\epsilon), e^{-\phi -\psi} |w_n|^{2\ell})$. 
Therefore,
$$
\bar \partial (u_{\delta} e^{\psi}) = e^{\psi} (v  + u _{\delta} \bar \partial \psi) 
$$
and 
$$
\int_{\mathbb{D}^{n-1}(2\epsilon) \times \mathbb{D}^* (\epsilon)} |u_{\delta} e^{\psi}|^2 e^{-\psi-\phi} |w_n|^{2\ell} dV \leq \int_{\mathbb{D}^{n-1}(2\epsilon) \times \mathbb{D}^* (\epsilon)} |e^{\psi} (v+ u_{\delta} \bar \partial \psi ) |^2_{i \partial \bar \partial (\psi + \phi)} e^{-\psi - \phi} |w_n|^{2\ell} dV.
$$
By this reason, for any $r > 0 $, it follows that
\begin{equation*}
\begin{aligned}
&\int_{\mathbb{D}^{n-1}(2\epsilon) \times \mathbb{D}^* (\epsilon) } |u_{\delta}|^2 e^{\psi- \phi} |w_n|^{2\ell} dV \\
& \leq \int_{\mathbb{D}^{n-1}(2\epsilon) \times \mathbb{D}^* (\epsilon)} \langle v+ u_{\delta} \bar \partial \psi, v + u_{\delta} \bar \partial \psi \rangle_{i \partial \bar \partial (\psi + \phi)} e^{-\psi - \phi} |w_n|^{2\ell} dV\\
&\leq \int_{\mathbb{D}^{n-1}(2\epsilon) \times \mathbb{D}^* (\epsilon)} |v|^2_{i \partial \bar \partial (\psi + \phi)} e^{\psi - \phi}|w_n|^{2\ell} dV + 2 \text{Re} \int_{\mathbb{D}^{n-1}(2\epsilon) \times \mathbb{D}^*(\epsilon)} \langle v , u_{\delta} \bar \partial \psi \rangle_{i \partial \bar \partial (\psi + \phi)} e^{\psi- \phi}|w_n|^{2\ell} dV  \\
&\quad + \int_{\mathbb{D}^{n-1}(2\epsilon) \times \mathbb{D}^*(\epsilon)} |u_{\delta}|^2 |\bar \partial \psi |^2_{i \partial \bar \partial (\psi+ \phi)} |w_n|^{2\ell} dV\\
&\leq \bigg( 1+ \frac{1}{r} \bigg) \int_{\mathbb{D}^{n-1} (2\epsilon) \times \mathbb{D}^* (\epsilon) } |v|^2_{i \partial \bar \partial (\psi + \phi)} e^{\psi - \phi}|w_n|^{2\ell} dV + \int_{\mathbb{D}^{n-1}(2\epsilon) \times \mathbb{D}^* (\epsilon) } |\bar \partial \psi|^2_{\partial \bar \partial (\phi + \psi)} |u_{\delta}|^2 e^{\psi - \phi}|w_n|^{2\ell} dV \\
&\quad + r \int_{\textbf{supp} v} | \bar \partial \psi|^2_{\partial \bar \partial (\phi + \psi)} |u_{\delta}|^2 e^{\psi - \phi}|w_n|^{2\ell} dV\\
&=: \left(1+\frac{1}{r}\right) I+II+ r III
\end{aligned}
\end{equation*}  
using the Cauchy-Schwarz inequality and arithmetic–geometric mean. 

By \eqref{complex hessian} and
$$
\partial \eta \wedge \bar \partial \eta = (1 + (-\rho)^{-1})^{2} \partial \rho \wedge \bar \partial \rho ,
$$
we know that
$$
\partial \bar \partial \psi \geq \frac{\partial \rho \wedge \bar \partial \rho}{\eta \rho^2} + \frac{\partial \eta \wedge \bar \partial \eta}{\eta^2} = \bigg( \frac{1}{\eta^2} + \frac{1}{\eta (-\rho +1)^2} \bigg) \partial \eta \wedge \bar \partial \eta = \bigg( 1 + \frac{\eta}{ (-\rho +1)^2} \bigg)  \partial \psi \wedge \bar \partial \psi  . 
$$
Therefore,
\begin{equation}\label{II}
\begin{aligned}
II&= \int_{\mathbb{D}^{n-1} (2\epsilon) \times \mathbb{D}^*(\epsilon)} |\bar \partial \psi|^2_{\partial \bar \partial (\phi + \psi)} |u_{\delta}|^2 e^{\psi - \phi} |w_n|^{2\ell} dV\\ &\leq \int_{\mathbb{D}^{n-1}(2\epsilon) \times \mathbb{D}^* (\epsilon)} \frac{|u_{\delta}|^2}{1+ \frac{\eta}{(-\rho+1)^2}} e^{\psi-\phi}|w_n|^{2\ell} dV.
\end{aligned}
\end{equation}
Since
$$
\partial \psi \wedge \bar \partial \psi = \eta^{-2}( 1 + (-\rho)^{-1})^2 \partial \rho \wedge \bar \partial \rho \leq \frac{4}{\eta^2} \partial \rho \wedge \bar \partial \rho
$$
and 
\begin{equation*}
\begin{aligned}
\partial \bar \partial \psi \geq   (1+(-\rho)^{-1}) \frac{\partial \bar \partial \rho}{\eta} &\geq \frac{ \delta^2 }{ \eta (| w_n |^2 + \delta^2)^2} \partial w_n \wedge \bar \partial w_n  \\
& \geq \frac{ | w_n |^2 }{\eta (| w_n |^2 + \delta^2)^2} \partial w_n \wedge \bar \partial w_n = \frac{\partial \rho \wedge \bar \partial \rho}{\eta}
\end{aligned}
\end{equation*}
on $\textbf{supp} v$,
we obtain that
\begin{equation}\label{III}
III = \int_{\textbf{supp} v} | \bar \partial \psi|^2_{\partial \bar \partial (\phi + \psi)} |u_{\delta}|^2 e^{\psi-\phi} |w_n|^{2\ell} \leq \int_{\mathbb{D}^n(2\epsilon) \times \mathbb{D}^{*}(\epsilon)} \frac{4}{\eta} |u_{\delta}|^2 e^{\psi-\phi} |w_n|^{2\ell} dV.
\end{equation}
Now, let us consider I.
We have
\begin{equation*}
\begin{aligned}
&I = \int_{\mathbb{D}^{n-1}(2\epsilon) \times \mathbb{D}^{*}(\epsilon) } |v|^2_{i \partial \bar \partial (\psi + \phi)} e^{\psi - \phi} |w_n|^{2\ell} dV\\
&= \int_{\mathbb{D}^{n-1}(2\epsilon) \times \mathbb{D}^{*}(\epsilon) } \bigg| \alpha \chi \bigg( \frac{| w_n |^2}{ \delta ^2} \bigg) \bigg|^2_{i \partial \bar \partial (\psi + \phi)} e^{\psi - \phi} |w_n|^{2\ell} dV \\
&\leq \int_{ \mathbb{D}^{n-1}(2\epsilon) \times \{ 0 <  |w_n|^2 <  \frac{ {3}  }{4}  \delta^2  \}   }  |\alpha|^2_{i \partial \bar \partial (\psi+ \phi)}  e^{\psi-\phi} |w_n|^{2\ell} dV\\
&\leq  \int_{ \mathbb{D}^{n-1}(2\epsilon) \times \{ 0 <  |w_n|^2 <  \frac{ {3}  }{4}  \delta^2  \} }  |\alpha|^2_{i \partial \bar \partial (\psi +\phi) } \frac{|w_n|^{2\ell}}{\eta} \frac{1}{| w_n |^2 + \delta^2}  e^{-\varphi} e^{-|w'|^2}  dV,
\end{aligned}
\end{equation*}
since 
$$
e^{\psi-\phi} = e^{-\log \eta} \cdot e^{-(|w'|^2 + \varphi + \log(|w_n|^2 + \delta^2) )} = \frac{1}{\eta} \frac{e^{-\varphi}  e^{-|w'|^2}  }{|w_n|^2 + \delta^2 }
$$
 by the definition of $\psi$ and $\phi$.
Moreover, by 
$$
\partial \bar \partial (\phi + \psi) \geq  \partial \bar \partial |w'|^2 + (1+(-\rho)^{-1}) \frac{\partial \bar \partial \rho}{\eta} \geq \sum_{j=1}^{n-1} d w_j \wedge d \bar w_j + \frac{\delta^2}{\eta (| w_n |^2 + \delta^2)^2} d w_n \wedge d \bar w_n,
$$
it follows that
\begin{equation}\label{upper norm of alpha} 
|\alpha|^2_{i \partial \bar \partial (\psi+\phi)} \leq  \frac{\eta (| w_n |^2 + \delta^2)^2 }{ \delta^2 }  |\alpha_n |^2  + \sum_{j=1}^{n-1} |\alpha_j|^2   .
\end{equation}
Hence by \eqref{upper norm of alpha} 
\begin{equation}\label{I}
\begin{aligned}
I &\leq  \int_{ \mathbb{D}^{n-1}(2\epsilon) \times  \{ 0 <  |w_n|^2 <  \frac{ {3}  }{4}  \delta^2  \}  }   |\alpha|^2_{i \partial \bar \partial (\psi +\phi) } \frac{|w_n|^{2\ell}}{\eta} \frac{1}{| w_n |^2 + \delta^2}  e^{-\varphi} e^{-|w'|^2}  dV \\
& \leq \int_{ \mathbb{D}^{n-1}(2\epsilon) \times \{ 0 <  |w_n|^2 <  \frac{ {3}  }{4}  \delta^2  \}  }  |\alpha_n |^2 {|w_n|^{2\ell}}\frac{|w_n|^2 + \delta ^2}{ \delta^2}  e^{-\varphi} e^{-|w'|^2}  dV   \\
&\quad + \sum_{j=1}^{n-1} \int_{ \mathbb{D}^{n-1}(2\epsilon) \times  \{ 0 <  |w_n|^2 <  \frac{ {3}  }{4}  \delta^2  \}}   |\alpha_j |^2 |w_n|^{2\ell} \frac{1}{ \eta (| w_n |^2+\delta^2) } e^{-\varphi}  e^{-|w'|^2} dV  \\
&\leq 2 \int_{\mathbb{D}^{n-1}(2 \epsilon) \times \mathbb{D}^* (\epsilon)} |\alpha_n |^2 |w_n|^{2\ell} e^{-\varphi}  e^{-|w'|^2} dV + \sum_{j=1}^{n-1} \int_{\mathbb{D}^{n-1}(2 \epsilon) \times \mathbb{D}^* (\epsilon) }  |\alpha_j |^2 |w_n|^{2\ell}   
\frac{e^{-\varphi} e^{-|w'|^2}}{ \eta (| w_n |^2+\delta^2) }   dV.
\end{aligned}
\end{equation} 
The third inequality holds because 
$\frac{|w_n|^2+\delta^2}{\delta^2} \leq 2$ when $0<|w_n|^2 < \frac{3}{4}\delta^2$.

As a result, by \eqref{I}, \eqref{II} and \eqref{III} we obtain
\begin{equation*}
\begin{aligned}
&\int_{\mathbb{D}^{n-1}(2\epsilon)\times \mathbb{D}^*(\epsilon)} \bigg( \frac{\frac{\eta }{(-\rho +1)^2 } } { 1 + \frac{\eta}{(-\rho +1)^2}} - \frac{4 r} {\eta} \bigg) |u_{\delta}|^2 e^{\psi- \phi} |w_n|^{2\ell} dV \\
& \leq \bigg( 1+ \frac{1}{r} \bigg) \bigg( 2 \int_{\mathbb{D}^{n-1}(2 \epsilon) \times \mathbb{D}^* (\epsilon)} |\alpha_n |^2  |w_n|^{2\ell} e^{-\varphi} {  e^{-|w'|^2} }  dV  
+\sum_{j=1}^{n-1} \int_{\mathbb{D}^{n-1}(2 \epsilon) \times \mathbb{D}^* (\epsilon) }  |\alpha_j |^2 \frac{|w_n|^{2\ell}e^{-\varphi}e^{-|w'|^2} }{ \eta (| w_n |^2+\delta^2) }   dV \bigg).
\end{aligned}
\end{equation*}  

Since
\begin{equation}\nonumber
\frac{\eta}{-\rho} = 1 + \frac{\log (-\rho)}{-\rho} \rightarrow 1,
\end{equation}
as $|w_n| \rightarrow 0$ and $\delta \rightarrow 0$, by taking a sufficiently small $r>0$ and $\epsilon ' >0 $, we obtain a positive constant $C_{r, \epsilon'} >0$ such that  
\begin{equation*}
\begin{aligned}
&\int_{\mathbb{D}^{n-1}(2\epsilon) \times {\mathbb{D}}^* (\epsilon ' ) }  \frac{|w_n|^{2\ell} } { (| w_n |^2 + \delta^2) (-\log (| w_n |^2 + \delta^2))^2} |u_{\delta}|^2 e^{-\varphi} dV \\
& \leq C_{r,\epsilon'} \bigg(  \int_{\mathbb{D}^{n-1}(2 \epsilon) \times \mathbb{D}^* (\epsilon ')} |\alpha_n |^2 e^{-\varphi} |w_n|^{2\ell} dV  
+\sum_{j=1}^{n-1} \int_{\mathbb{D}^{n-1}(2 \epsilon) \times \mathbb{D}^* (\epsilon ') }  |\alpha_j |^2 \frac{e^{-\varphi}|w_n|^{2\ell}}{ -\rho (| w_n |^2+\delta^2) }  dV\bigg).
\end{aligned}
\end{equation*}
This inequality holds, since $e^{-|w'|^2}$ has a positive upper and lower bound on $\mathbb{D}^{n-1}(2\epsilon)$. 

If we set $\varphi = - s \log (-\rho) + \tau$,  then it is still plurisubharmonic when $s \geq 0$ and $e^{-\varphi}=(-\rho)^{s}e^{-\tau}$. So,
\begin{equation*}
\begin{aligned}
&\int_{\mathbb{D}^{n-1}(2\epsilon) \times \mathbb{D}^* (\epsilon ' )}  \frac{ |w_n|^{2\ell} } { (| w_n |^2+\delta^2)} (-\log (|w_n|^2 + \delta^2) )^{s-2} |u_{\delta}|^2 e^{-\tau} dV \\
& \leq C_{r,\epsilon'} \bigg(  \int_{\mathbb{D}^{n-1}(2 \epsilon) \times \mathbb{D}^* (\epsilon ')} |\alpha_n |^2 (-\rho)^{s} |w_n|^{2\ell} e^{-\tau} dV  
+\sum_{j=1}^{n-1} \int_{\mathbb{D}^{n-1}(2 \epsilon) \times \mathbb{D}^* (\epsilon ') }  |\alpha_j |^2 \frac{ |w_n|^{2\ell}}{ | w_n |^2+\delta^2  } (-\rho)^{s-1} e^{-\tau} dV \bigg) \\
&\leq 2 C_{r,\epsilon'} \bigg(  \int_{\mathbb{D}^{n-1}(2 \epsilon) \times \mathbb{D}^* (\epsilon ')} |\alpha_n |^2 (- \log |w_n| )^{s} |w_n|^{2\ell} e^{-\tau} dV  \\
&\quad \quad +\sum_{j=1}^{n-1} \int_{\mathbb{D}^{n-1}(2 \epsilon) \times \mathbb{D}^* (\epsilon ') }  |\alpha_j |^2  \frac{1}{| w_n |^2 } |w_n|^{2\ell} (-\log |w_n|)^{s-1} e^{-\tau} dV \bigg)
\end{aligned}
\end{equation*}
Here, we use the facts that $x^{s}$ is increasing when $s > 0$ and $(x (-\log x))^{-1}$ is decreasing when $0<x<e^{-1}$.  Therefore, since
$
\| w \|^2 \sim |w_n|^2
$
and
$
-\log \| w \| \sim - \log |w_n|
$
on $\mathbb{D}^{n}(2\epsilon)$, the lemma is proved when $s \geq 0$ by letting $\delta \rightarrow 0$ and using a weak limit argument. 
\end{proof}

\begin{proposition}
Let $\Gamma$ be a torsion-free lattice of $\text{Aut}(\mathbb{B}^n)$ with only unipotent parabolic automorphisms and $\Sigma = \mathbb{B}^n/ \Gamma$ be a ball quotient with finite volume. Let $F=S^m T_{\Sigma}^*$. If $m \geq n-1$, then for every $r \geq 0$, the following sequence
\begin{equation}\label{local dolbeault for (r,0) 2}
0 \rightarrow \ker \bar\partial_{(r,0),F}  \xrightarrow {} L^{(r,0)}_{2,F} \xrightarrow {\bar \partial} \ker \bar \partial_{(r,1),F} \rightarrow 0
\end{equation}
is exact.
\end{proposition}

\begin{proof}

For any open set in $\Sigma$, the exactness of \eqref{local dolbeault for (r,0) 2} is well known and hence we only need to consider open sets intersecting with the corresponding divisor $D_{b}$ in $\overline{\Sigma}$ of a cusp $b$ of $\Gamma$. Consider a divisor $D_{b} \subset \overline \Sigma$ and fix a point $p \in D_{b}$. {Take a sufficiently small polydisc } $p \in U_p = \mathbb{D}^{n-1}(\epsilon) \times \mathbb{D}^* (\epsilon) \subset \subset \Omega_{b}^{(N)}$ of $q$ such that there exists a holomorphic coordinate system $(w_1,\cdots,w_n)$ on an open set $W_p$ containing $\overline{U_p}$ which satisfies the statements of Lemma~\ref{quasi-isometric} and \ref{comparison between norms}. 

To verify the exactness of \eqref{local dolbeault for (r,0) 2}, it suffices to show that there exists an open set $p \in V_p \subseteq U_p$ such that the following $\bar \partial$-equation 
\begin{equation}\label{system for (r,1) 1}
\bar \partial u = f, \quad \bar \partial f=0 ~~ \text{and} ~~ f \in L^{(r,1)}_{2} (U_p, F)
\end{equation}
has a solution $u \in L^{(r,0)}_{2,F} (V_p) $ satisfying 
\begin{equation}\nonumber
\int_{V_p - D_b} |u|^2_{h^F, \omega} dV_{\omega} \leq C \int_{U_p - D_b} |f|^2_{h^F,\omega} dV_{\omega}
\end{equation}
for some constant $C>0$.

Now, write
$$
u = \sum_{\substack{|I|=m \\ j_1 < \cdots < j_r } } u_{j_1 \cdots j_r , I} dw_{j_1} \wedge \cdots \wedge dw_{j_r} \otimes  e_I$$
and 
$$
f = \sum_{\substack{|I|=m \\ j_1 < \cdots < j_r }} \left( \sum_{\ell=1}^{n} f_{j_1 \cdots j_r, I,\ell} dw_{j_1} \wedge \cdots dw_{j_r} \wedge d\bar w_\ell \right) \otimes e_I
$$
where $e^{I}=dw_1^{i_1} \cdots dw_n^{i_n}$.  Since 
$\omega \sim \widetilde \omega$ on $W_q$, 
$
|u|^2_{\omega} \sim \big| \sum_{|I|=m} u_I e^I     \big |^2_{\widetilde \omega}
$ and so by Lemma \ref{quasi-isometric}

\begin{equation}\label{norm comparison for (r,s) 1}
\begin{aligned}
|u|^2_{\omega} &\sim  \sum_{\substack{|I|=m \\ j_1 < \cdots < j_r} } |u_{j_1 \cdots j_r, I} |^2  (dw_{j_1} , dw_{j_1} )_{\widetilde \omega} \cdots (dw_{j_r}, dw_{j_r})_{\widetilde \omega} \cdot  (dw_1, dw_1)^{i_1}_{\widetilde \omega} \cdots (dw_n, dw_n)^{i_n}_{\widetilde \omega} \\
&\sim \sum_{\substack{|I|=m \\ j_1 < \cdots < j_r} } |u_{j_1 \cdots j_r, I } |^2 ( dw_{j_1} , dw_{j_1} )_{\widetilde \omega} \cdots (dw_{j_r}, dw_{j_r})_{\widetilde \omega} \cdot  \| w \|^{2i_n} (-\log \| w \|)^{m+i_n}.
\end{aligned}
\end{equation}
Similarly,
\begin{equation}\label{norm comparison for (r,s) 2}
\begin{aligned}
|f|^2_{\omega} &\sim \sum_{\substack{|I|=m \\ j_1 < \cdots < j_r}} \sum_{\ell=1}^{n} |f_{j_1 \cdots j_r, I,\ell}|^2  ( dw_{j_1} , dw_{j_1} )_{\widetilde \omega} \cdots (dw_{j_r}, dw_{j_r})_{\widetilde \omega} \cdot (d \bar w_\ell, d \bar w_\ell)_{\widetilde \omega} (e_I, e_I)_{\widetilde \omega} \\
&\sim \sum_{\substack{|I|=m \\ j_1 < \cdots < j_r} } \bigg( \sum_{\ell=1}^{n-1} |f_{j_1 \cdots j_r, I,\ell}|^2 \|w \|^{2 i_n} (-\log \|w \|)^{m+i_n+1}  \\
& \quad \quad +  |f_{j_1 \cdots j_r, I, n}|^2 \|w \|^{2 (i_n + 1)} (-\log \|w \|)^{m+i_n +2 } \bigg) \cdot ( dw_{j_1} , dw_{j_1} )_{\widetilde \omega} \cdots (dw_{j_r}, dw_{j_r})_{\widetilde \omega}  .
\end{aligned}
\end{equation} 
Hence  by \eqref{norm comparison for (r,s) 1}, \eqref{norm comparison for (r,s) 2}, and $$
dV_{\omega} = \frac{dV}{\| w\|^2 (-\log \| w \|)^{n+1}}
$$
where $dV$ is the standard Lebesgue measure, we know that to find a solution $u$ of the $\bar \partial$-equation \eqref{system for (r,1) 1}, it suffices for each fixed $I$ and $j_1 < \cdots < j_r$ to find a solution $u_I$ on $V_p - D_b$ of the following $\bar \partial$-equation 
\begin{equation}\label{system for (r,s) 2}
\bar \partial u_{j_1 \cdots j_r ,I}  = \sum_{\ell=1}^{n} \frac{ \partial u_{j_1\cdots j_r, I}}{\partial \bar w_\ell} d \bar w_\ell  = \sum_{\ell=1}^{n} f_{j_1 \cdots j_r, I,\ell} d \bar w_\ell \quad \text{provided} \quad \bar \partial \bigg( \sum_{\ell=1}^{n} f_{j_1 \cdots j_r, I, \ell} d \bar w_\ell   \bigg)=0
\end{equation}
which satisfies 
\begin{equation*}\label{equivalnce of integral for (r,s) 1}
\begin{aligned}
& \int_{V_p - D_b} |u_{j_1 \cdots j_r, I} |^2 \bigg( \|w \|^{2 (i_n-1) } (-\log \|w \|)^{m+i_n-(n+1)} \bigg) \cdot ( dw_{j_1} , dw_{j_1} )_{\widetilde \omega} \cdots (dw_{j_r}, dw_{j_r})_{\widetilde \omega}  dV \\
&\leq \sum_{\ell=1}^{n-1} \int_{U_p - D_b}  |f_{j_1 \cdots j_r, I,\ell}|^2 \|w \|^{2 (i_n-1)} (-\log \|w \|)^{m+i_n-n}  \cdot ( dw_{j_1} , dw_{j_1} )_{\widetilde \omega} \cdots (dw_{j_r}, dw_{j_r})_{\widetilde \omega}  dV \\
&\quad\quad\quad\quad + \int_{U_p - D_b}  |f_{j_1 \cdots j_r, I, n}|^2 \|w \|^{2 i_n } (-\log \|w \|)^{m+i_n -(n-1) }  \cdot ( dw_{j_1} , dw_{j_1} )_{\widetilde \omega}  \cdots (dw_{j_r}, dw_{j_r})_{\widetilde \omega}  dV.
\end{aligned}
\end{equation*}
Therefore, the solvability of \eqref{system for (r,s) 2} follows by Lemma \ref{local L2 estimate}, by taking $s:=(m+i_n)-(n-1)$ and 
\begin{equation*}
\tau :=     
\left \{ \begin{array}{ll} 
- r \log( - \log \| w \|)  & \text{if $j_r \not = n$},\\
-(r+1) \log (- \log \| w \|) - \log |w_n |^2 & \text{if $j_r = n $}
\end{array} \right.
\end{equation*} 
in the Lemma.
\end{proof}

\begin{corollary}\label{exactness 1}
Let $\Gamma$ be a torsion-free lattice of $\text{Aut}(\mathbb{B}^n)$ with only unipotent parabolic automorphisms and $\Sigma = \mathbb{B}^n/ \Gamma$ be a ball quotient with finite volume. Let $F=S^m T_{\Sigma}^*$. If $m \geq n-1$, then for any  $1 \leq r \leq n-1$, the following sequences
\begin{equation*}
0 \rightarrow \ker \bar\partial_{(r,0),F}  \xrightarrow {} L^{(r,0)}_{2,F} \xrightarrow {\bar \partial} \ker \bar \partial_{(r,1),F} \rightarrow 0
\end{equation*}
are exact.
\end{corollary}

\begin{proof}[Proof of Theorem~\ref{main}] 
{ For each fixed $r, m$, let $E:=E_{r,m}$.} By Lemma \ref{exactness for (r,s) 0}, Corollary \ref{exactness 1}, and Proposition \ref{exactness for (r,s) 2} it follows that 
\begin{equation*}
0 \rightarrow \mathcal{O}(E_{r,m}) \xrightarrow { } L^{r,0}_{2, S^m T_{\overline \Sigma}^*} \xrightarrow {\bar \partial} \ker \bar  \partial_{(r,1), S^m T_{\overline \Sigma}^* } \rightarrow 0
\end{equation*}
and
\begin{equation*}
0 \rightarrow \ker \bar\partial_{(r,s), S^m T_{\overline \Sigma}^* }  \xrightarrow { } L^{r,s}_{2,S^m T_{\overline \Sigma}^*} \xrightarrow {\bar \partial} \ker \bar \partial_{(r,s+1), S^m T_{\overline \Sigma}^*} \rightarrow 0
\end{equation*}
are exact for every $0 \leq s \leq n$. It means that the resolution 
\begin{equation*}\label{exactnees 3}
0 \rightarrow \mathcal{O}( E)  \rightarrow L^{0,*}_{2, S^m T_{\overline \Sigma}^* }  
\end{equation*}
is exact. Therefore, by Lemma \ref{sheaf cohomology}, the theorem is proved.
\end{proof}

\section{Application}\label{application}

In this section, we prove a version of $L^2$-holomorphic jet extension theorem for a complex hyperbolic space form. For this,  we first establish some notation. Let $K_{\mathbb{B}^n}$ denote the Bergman kernel of $\mathbb{B}^n$, given by
$$
K_{\mathbb{B}^n} (z,w) = \frac{1}{(1-z \cdot \bar w)^{n+1}}
$$ 
and let its associated K\"ahler form 
$$
G =\frac{1}{n+1} \sqrt{-1} \partial \bar \partial \log K_{\mathbb{B}^n} (z,z).
$$  
We now consider an automorphism of $\mathbb{B}^n$,
$$
T_z (w) = \frac{z - P_z (w) - s_z Q_z(w)}{1- w \cdot \bar z},
$$
where $|z|^2 = z \cdot \bar z$ and $s_z = \sqrt{1-|z|^2}$, $P_z$ is the orthogonal projection from $\mathbb{C}^n$ onto the one-dimensional subspace $[z]$ generated by $z$, and $Q_z$ is the orthogonal projection from $\mathbb{C}^n$ onto $[z]^{\perp}$.  We have $T_z \circ T_z = \text{Id}_{\mathbb{B}^n}$.  Let $A= (A_{jk}) := dT_z (z)$, and define
$$
e_j := \sum_{k=1}^{n} A_{jk} d \bar z_k
$$
Then, $\{e_1, \cdots, e_n\}$ forms an orthonormal frame of $T_{\mathbb{B}^n}^*$.  We denote the Laplacian 
$$
\Box_m^{\ell} : C^{\infty} (\Sigma, S^m T_{\Sigma}^* \otimes \Lambda^{0,\ell} T_{\Sigma}^*) \rightarrow C^{\infty} (\Sigma, S^m T_{\Sigma}^* \otimes \Lambda^{0,\ell} T_{\Sigma}^*)
$$ 
by 
$$
\Box_m^{\ell} = \bar \partial \circ \bar \partial^* + \bar \partial^* \circ \bar \partial
$$
and we define $\mathcal{R}_{G}^m$ by
 \begin{equation}\nonumber
\begin{aligned}
 \mathcal R_{G}^m : C^{\infty} \left(\Sigma, S^{m}  T^{*}_{\Sigma} \right) &\rightarrow C^{\infty} \left( \Sigma, S^{m+1} T^{*}_{\Sigma}  \otimes \Lambda^{0,1}  T_{\Sigma}^{*}  \right)\\
u= \sum_{J} u_{J} e^{J}  &\mapsto \sum_{J,\ell} (u_{J} e^{J} e_\ell ) \otimes \bar  e_\ell
\end{aligned}
\end{equation}
where $e^{J} e_\ell$ is the symmetric product of $e^J$ and $e_\ell$. For the details of $\mathcal{R}_{G}^m$, see Section 3 of \cite{LS23-2}.

Let $h^{S^m T_{\Sigma}^*}$ be the metric on $S^m T_{\Sigma}^*$ induced from $G$. For the simplicity of notation, we denote 
$$
\langle u, v \rangle := \langle u, v \rangle_{h^{S^m T_{\Sigma}^*},G} \quad \text{and} \quad \langle \langle u, v \rangle \rangle := \langle \langle u, v \rangle \rangle_{h^{S^m T_{\Sigma}^*}, G}
$$
for any $u, v \in C^{\infty}_{0,s}(\Sigma, S^m T_{\Sigma}^*) \cap L^{0,s}_{2} (\Sigma, S^m T_{\Sigma}^*),~~ 0 \leq s \leq n$. 

\begin{lemma}\label{adjoint of RG}
Let $\mathcal{R}_{G,m}^{*}$ be the {adjoint of} $\mathcal{R}_{G}^{m}$ satisfying
$$
\langle \langle \mathcal{R}_{G,m}^* u_0, u_1  \rangle \rangle_{} = \langle \langle u_0, \mathcal{R}_{G}^m u_1 \rangle \rangle_{}, \quad \forall u_\ell \in {C_{0,\ell}^{\infty} } (\Sigma, S^m T_{\Sigma}^*) \cap L^{0,\ell}_2 (\Sigma, S^m T_{\Sigma}^*), \; \ell=0,1.
$$
Then, for $u=\sum_{|I|=m} u_{I} e^{I}$ and $v= \sum_{|J|=m+1, \ell=1}^{n} v_{J \ell} e^{J} \otimes \bar e_{\ell}$, we have
$$
\mathcal{R}^*_{G,m} v = \sum v_{J \ell } \mu_{\ell} (e^J)
$$
where $\mu_{\ell}(e^J) := e_1^{j_1} \cdots e_{\ell}^{j_{\ell}-1} \cdots e_n^{j_n}$.
In particular, 
$$
\mathcal{R}_{G,m}^* : L^{0,1}_2 (\Sigma, S^m T_{\Sigma}^*) \cap {C^{\infty}_{0,1}}  (\Sigma, S^m T_{\Sigma}^*) \rightarrow L^{0,0}_2 (\Sigma, S^m T_{\Sigma}^*) \cap {  C^{\infty}} (\Sigma, S^m T_{\Sigma}^*).
$$
\end{lemma}
\begin{proof}
We consider 
\begin{equation*}
\begin{aligned}
\langle \mathcal{R}_{G}^m u, v \rangle 
&= \bigg \langle \sum_{|I|=m} \sum_{k=1}^{n} u_I e^I e_k \otimes \bar e_k , \sum_{|J|=m+1} \sum_{\ell=1}^{n} v_{J\ell} e^{J} \otimes \bar e_{\ell} \bigg \rangle \\&= \sum_{\ell=1}^{n}  \bigg \langle \sum_{|I|=m} u_I e^I e_\ell , \sum_{|J|=m+1} v_{J\ell} e^{J} \bigg \rangle \\
&= \sum_{\ell=1}^{n} \sum_{|I|=m}  u_{I} v_{i_1 \cdots (i_\ell + 1) \cdots i_n, \ell}.
\end{aligned}
\end{equation*}
Since
\begin{equation*}
\begin{aligned}
\bigg \langle \sum_{|I|=m}^{} u_I e^I, \sum_{|J|=m+1} \sum_{\ell=1}^{n} v_{J\ell} \mu_{\ell} (e^{J})  \bigg \rangle &= \bigg \langle \sum_{|I|=m}^{} u_I e^I, \sum_{|J|=m+1} \sum_{\ell=1}^{n} v_{J\ell} e_1^{j_1}\cdots e_\ell^{j_\ell -1} \cdots e_n^{i_n} \bigg \rangle \\
&= \sum_{\ell=1}^{n} \sum_{|I|=m} u_{I} v_{i_1 \cdots (i_\ell +1) \cdots i_n, \ell},
\end{aligned}
\end{equation*}
we obtain
$$
\langle \langle \mathcal{R}_{G}^m u, v \rangle \rangle = \langle \langle u, \sum v_{J\ell} \mu_{\ell}(e^J) \rangle \rangle.
$$ 
Hence the lemma is proved.
\end{proof}

\begin{lemma}\label{property of RG}
Let $\ker^{\perp}(\Box_{m}^{\ell} - \lambda I), \ell=0,1$ be the orthogonal complement in $L^{0,\ell}_{2}(\Sigma, S^m T_{\Sigma}^*)$. Then 
\begin{enumerate}
\item $\mathcal{R}_{G}^{m} (\ker (\Box_{m}^{0} -\lambda I) ) \subset \ker (\Box_{m+1}^{1} - (\lambda + 2m)I) $
\item $\mathcal{R}_{G}^{m} ( \ker^{\perp} (\Box_{m}^{0} - \lambda I) \cap C^{\infty} (\Sigma, S^m T_{\Sigma}^*) ) \subset \ker^{\perp} (\Box_{m+1}^{1} - (\lambda + 2m) I)$. 
\end{enumerate}
\end{lemma}
\begin{proof}
The first assertion follows directly. To prove the second assertion, it suffices to prove 
$$
\mathcal{R}_{G}^{m} (\ker^{\perp} (\Box_{m}^{0} - \lambda I) \cap C^{\infty} (\Sigma, S^m T_{\Sigma}^*) ) \perp \ker(\Box_{m+1}^{1} - (\lambda + 2m) I).
$$
Take any $f  \in  \ker(\Box_{m+1}^{1} - (\lambda + 2m) I)$.

Let $\mathcal{R}_{G,m}^{*}$ be the adjoint of $\mathcal{R}_G^m$ in Lemma \ref{adjoint of RG}. 
By the self-adjointness of $\Box$ and Proposition 3.6 in \cite{LS23-2}, 
\begin{equation*}
\begin{aligned}
\langle \langle \Box^{0}_m \mathcal{R}_{G,m}^* f, u \rangle \rangle 
&= \langle \langle \mathcal{R}_{G,m}^* f, \Box^{0}_m u \rangle \rangle = \langle \langle f, \mathcal{R}_{G}^{m} \Box^0_m u \rangle \rangle \\
&=  \langle \langle  f, \Box^1_{m+1} \mathcal{R}_{G}^m u - 2m \mathcal{R}_{G}^m u \rangle \rangle \\
&= \langle \langle f , \Box^1_{m+1} \mathcal{R}_{G}^m u \rangle \rangle -2m \langle \langle  f , \mathcal{R}_{G}^m u \rangle \rangle \\
& = \langle \langle \Box^{1}_{m+1} f, \mathcal{R}_{G}^m u \rangle \rangle - 2m \langle \langle f, \mathcal{R}_{G}^m u \rangle \rangle \\
&= \langle \langle \lambda f, \mathcal{R}_{G}^{m} u \rangle \rangle = \langle \langle \lambda \mathcal{R}_{G,m}^* f, u \rangle \rangle
\end{aligned}
\end{equation*}
for every $u \in C^{\infty} (\Sigma, S^m T_{\Sigma}^*) \cap L^{0,0}_2 (\Sigma, S^m T_{\Sigma}^*)$. Hence
$
\Box^0_m \mathcal{R}_{G,m}^* f = \lambda \mathcal{R}_{G,m}^* f
$
and $\mathcal{R}_{G,m}^* f \in \ker (\Box_m^{0} - \lambda I)$. 
Thus,
$$
\langle \langle f, \mathcal{R}_{G}^m v \rangle \rangle = \langle \langle \mathcal{R}_{G,m}^* f, v \rangle \rangle = 0
$$
for every $v \in \ker^{\perp}(\Box_{m}^{0} - \lambda I)\cap { C^{\infty} }(\Sigma, S^m T_{\Sigma}^*)$. Therefore, the proof is completed.
\end{proof}
\begin{remark}
For a compact compact ball quotient, analogous properties given in Lemma~\ref{property of RG} are presented in \cite[Corollary 3.7]{LS23-2} and \cite[Corollary 2.3]{LS23-1}.
In these cases, the authors used the compactness of the Green operator to show (2) in Lemma~\ref{property of RG}, however for non-compact $\Sigma$ we have no idea whether it is compact.
\end{remark}

\begin{proof}[Proof of Theorem \ref{extension}]
Since the rest of the proof is similar to that of the case when $\Sigma$ is compact given in \cite{LS23-2}, we only present a sketch of the proof.  

The key ingredient of the proof is to construct a power series $\Phi(\psi) \in A^2_{\alpha}(\Omega)$ for a given $\psi \in H^{0,0}_{L^2, \bar \partial} (\Sigma, S^N T_{\Sigma}^*) $. For this, we note that any $f \in \mathcal{O}(\Omega)$ can be regarded as $f \in \mathcal{O}(\mathbb{B}^n \times \mathbb{B}^n)$ which is invariant under the diagonal action of $\Gamma$. By letting $t := T_z w$, we obtain a smooth function $\widetilde f (z,t) := f(z, T_z t) = f(z,w)$ which is holomorphic in $t$ but not in $z$. Since $w = T_z t$, from the Taylor expansion of $\widetilde f$ at $(z,0)$,  we have
\begin{equation}\nonumber
f(z,w) = \sum_{|I|=0}^{\infty} f_I (z) (T_z w)^I. 
\end{equation}
By \cite[Proposition 4.9]{LS23-2}, we obtain the \textit{associated differential} $\varphi$ of $f$ which is defined by
$$
 \varphi := \sum_{k=0}^{\infty} \varphi_k, \quad \text{ where }
\quad \varphi_{k} := \sum_{|I|=k} \varphi_I, \quad \varphi_I := f_I (z) e^{I}$$
and $\varphi_k$, $k \in \mathbb{N}$ satisfy
\begin{equation}\label{a recursive formula}
\bar \partial \varphi_ k = -(k-1) \mathcal{R}_{G} (\varphi_{k-1}) \; \;\text{on} \; \, \Sigma.
\end{equation}

Now, in a reverse way, we construct $\Phi(\psi)$ by using \eqref{a recursive formula} from a symmetric differential $\psi\in H^{0,0}_{L^2, \bar \partial} (\Sigma, S^N T_{\Sigma}^*)$. If $N=0$, by identifying $S^0 T_{\Sigma}^* \cong \Sigma \times \mathbb{C}$, we consider $\psi$ as an $L^2$-holomorphic function on $\Sigma$. Then, we define $\Phi(\psi) (z,w) := \widetilde \psi (z)$ where $\widetilde \psi$ is the lifting of $\psi$ by the quotient map $\mathbb{B}^n \rightarrow \Sigma$.

If $N \geq 1$,  then for a given $\psi \in H^{0,0}_{L^2, \bar \partial} (\Sigma, S^N T_{\Sigma}^*)$, we define inductively 
$ \{\varphi_k \}$
by
\begin{equation*}\label{system}
\left\{ \begin{array}{ll}
\varphi_{k}=0 & \text{if $k<N$}, \\
\varphi_{N}=\psi ,&
\end{array} \right.
\end{equation*}
and for $s \geq 1$,  $\varphi_{N+s}$ is the minimal solution of
the following $\overline \partial$-equation:
\begin{equation}\nonumber
\bar \partial \varphi_{N+s} = - (N+s -1) \mathcal R_G\left( \varphi_{N+s-1}  \right).
\end{equation}
By following the proof of \cite[Lemma 4.12]{LS23-2} and using the existence of the Green operator given in Theorem~\ref{Hodge decomposition} and (1) in Lemma \ref{property of RG}, we obtain 
\begin{equation}\label{solsuff}
\| \varphi_{N+s} \|^2 =  \bigg( \prod_{j=1}^{s} \left(1+ \frac{n-1}{N+j} \right) \bigg) \left( \frac{(2N-1)!}{ \{ (N-1)!  \}^2}
\frac{ \{ (N+s-1)! \}^2} { (2N+s-1)!}\frac{1}{s! } \right) \| \psi \|^2
\end{equation}
for any $s \geq 1$.
Let us express
\begin{equation*}
\varphi_m := \sum_{|I|=m} f_I(z) e^I
\end{equation*}
and we define a formal sum 
\begin{equation}\label{formal power series}
f(z,w) := \sum_{|I|=0}^{\infty} f_I(z) (T_z w)^{I}
\end{equation}
on $\Omega$.  Then, by using \eqref{solsuff} and following the proofs of \cite[Lemma 4.13, 4.14, 4.15, and 4.16]{LS23-2}, it follows that $f$ is a holomorphic function on $\Omega$. As a result, when $N \geq 1$, for a given $\psi \in H^{0,0}_{L^2, \bar \partial} (\Sigma,  S^N T_{\Sigma}^*)$, by using 
\eqref{formal power series}, 
we define
$
\Phi (\psi) := f
$
 and extend $\Phi$ linearly on $\bigoplus_{m=0}^{\infty} H^{0,0}_{L^2, \bar \partial} (\Sigma, S^m T_{\Sigma}^*)$. Then, $\Phi$ is a linear map. For the injectivity, see the proof of \cite[Lemma 4.17]{LS23-2}, and {the density property of $\Phi$ follows from the proof of Lemma 4.18 with (2) in Lemma \ref{property of RG}.}
\end{proof}

\begin{proof}[Proof of Corollary~\ref{no nonconstant holo}]
   Let $\Gamma'\subset \Gamma$ be a sublattice  of finite index such that $\Gamma'$ has only unipotent parabolic automorphisms. Let $\Omega'$ be the quotient of $\mathbb B^n\times\mathbb B^n$ by the diagonal action of $\Gamma'$. By the same argument given in the proofs of Corollary~4.19 and Theorem~4.20 in \cite{LS23-2}, we have $A^2_{-1}(\Omega')\cong \mathbb C$ and there exists no bounded holomorphic function on $\Omega'$. This implies that $\Omega$ also has the same properties.
\end{proof}

By a similar argument used in the proof of Theorem~1.5, it is possible to generalize Theorem~1.1 in \cite{LS23-1}:
Let $\widetilde{M}$ be a complex manifold, $\Gamma$ be a torsion-free lattice of $\text{Aut}(\widetilde{M})$ and $\rho\colon\Gamma\to SU(N,1)$ be a representation. Suppose that there exists a $\rho$-equivariant totally geodesic isometric holomorphic embedding $\imath\colon \widetilde M\to\mathbb B^N$. Let $\Omega_{\rho}:= M \times_{\rho} \mathbb{B}^N$ be a holomorphic $\mathbb{B}^N$-fiber bundle over $M= \widetilde M / \Gamma$, where any $\gamma \in \Gamma$ acts on $\widetilde M \times \mathbb{B}^N$ by  $(\zeta,w)\mapsto (\gamma \zeta, \rho(\gamma) w)$. We define a K\"ahler form $\omega$ on $\Omega_{\rho}$ by
\begin{equation}\nonumber
\omega|_{[\zeta, w]} = {\widetilde H}  +  \frac{\sqrt{-1}}{N+1} \partial \bar \partial \log K(w,w)
\end{equation}
with the K\"ahler form {$\widetilde H$ for $(\widetilde M, \imath^* g_{\mathbb{B}^N})$}, {where $\imath^* g_{\mathbb{B}^N}$ is the pull-back metric on $\widetilde M$ of the normalized Bergman metric $g_{\mathbb{B}^N}$ of $\mathbb{B}^N$.}
One can check that $\omega$ is an $(1,1)$ form on $\Omega_{\rho}$. We define the volume form on $\Omega_{\rho}$ by {$dV_{\omega} = \frac{1 }{(N+n)!}\omega^{N+n} $.} For measurable sections $f_1$, $f_2$ on $\Lambda^{r,s} T_{\Omega_{\rho}}^{*}$ and $\alpha > -1$,
we set
\begin{equation}\nonumber
\langle \langle f_1,f_2 \rangle \rangle_{\alpha} : = c_{\alpha} \int_{\Omega_{\rho}} \langle f_1, f_2 \rangle_{\omega} \delta^{\alpha+N+1} \,dV_\omega
\end{equation}
where $c_{\alpha} = \frac{\Gamma(N+\alpha+1)}{ \Gamma(\alpha+1) N! }$ and $\delta = 1-|T_{\imath (\zeta)} w|^2$.

For $\alpha > -1$, we define a weighted $L^2$-space by setting
\begin{equation}\nonumber
L^2_{(r,s), \alpha}(\Omega_{\rho}) := \{ f : f \text{ is a measurable section on $\Lambda^{r,s} T_{\Omega_{\rho}}^*$}, ~\|f \|^2_{\alpha} := \left< f,f\right>_{\alpha} < \infty \}
\end{equation}
and a weighted Bergman space by
$A^2_{\alpha}(\Omega_{\rho}) := L^2_{(0,0), \alpha}(\Omega_{\rho}) \cap \mathcal O(\Omega_{\rho}).$ In this setting, we extend $\bar \partial$-operator on $\Omega_{\rho}$ as the maximal extension of $\bar \partial$ on $\Omega_{\rho}$ which acts on smooth $(r,s)$ forms on $\Omega_{\rho}$. The \textit{Hardy space} $A^2_{-1}(\Omega_{\rho})$ is defined by
$$
A^2_{-1}(\Omega_{\rho}) := \{  f \in \mathcal{O}(\Omega_{\rho}) : \|f \|^2_{-1} :=  \lim_{\alpha \searrow -1} \| f \|^2_{\alpha} < \infty  \}.
$$

\begin{theorem}
Let $\widetilde{M}$ be a complex manifold, $\Gamma$ be a torsion-free lattice of $\text{Aut}(\widetilde{M})$ and $\rho\colon\Gamma\to SU(N,1)$ be a representation such that $\rho(\Gamma)$ has only unipotent parabolic elements.
Suppose that there exists a $\rho$-equivariant totally geodesic isometric holomorphic embedding $\imath\colon \widetilde M\to\mathbb B^N$ and the volume of $M := \widetilde M / \Gamma$ is finite for the induced metric from $\widetilde M$. Let $\Sigma_{\rho}:=\mathbb B^N/\rho(\Gamma)$ and $\Omega_{\rho}:=M\times_\rho \mathbb B^N$ be a holomorphic $\mathbb B^N$-fiber bundle over $M$ where any $\gamma\in \Gamma$ acts on $\widetilde M\times \mathbb B^N$ by $(\zeta,w)\mapsto (\gamma \zeta, \rho(\gamma) w)$. Then there exists an injective linear map
\begin{equation}\nonumber
\Phi: \bigoplus_{m=0}^{\infty} H^0 (M, \imath^* (S^m T_{\Sigma_{\rho}}^*)) \rightarrow
\begin{cases}
\displaystyle\bigcap_{\alpha>-1} A^2_\alpha (\Omega_{\rho}) \subset \mathcal{O} (\Omega_{\rho}) & \text{ if } n=N,\\
\displaystyle\bigcap_{\alpha\geq -1} A^2_\alpha (\Omega_{\rho}) \subset \mathcal{O} (\Omega_{\rho})  & \text{ if } n< N,\\
\end{cases}
\end{equation}
which has a dense image in $\mathcal{O}(\Omega_{\rho})$ equipped with the compact open topology. In particular, $\dim A^2_{\alpha} (\Omega_{\rho}) = \infty$ if $\alpha> -1$ and $A_{-1}^2 (\Omega_{\rho}) = \bigcap_{\alpha \geq -1} A_\alpha^2(\Omega_{\rho})$ with $\dim A^2_{-1} (\Omega_{\rho}) = \infty$ if $n < N$.
\end{theorem}

\end{document}